\documentclass[10pt]{amsart}
\usepackage{latexsym, amsmath,amssymb,  bbold}

	\usepackage{xcolor}
	
	\usepackage{graphicx}

\renewcommand{\epsilon}{\varepsilon}
\newcommand{\newsection}[1] 
{\subsection{#1}\setcounter{theorem}{0} \setcounter{equation}{0} 
\par\noindent}

\newtheorem{theorem}{Theorem}

\newtheorem{lemma}[theorem]{Lemma}
\newtheorem{corr}[theorem]{Corollary}

\newtheorem{proposition}[theorem]{Proposition}
\newtheorem{deff}[theorem]{Definition}

\newcommand{\bth}{\begin{theorem}}
\newcommand{\ble}{\begin{lemma}}
\newcommand{\bcor}{\begin{corr}}

\newcommand{\bdeff}{\begin{deff}}

\newcommand{\bprop}{\begin{proposition}}
\newcommand{\ele}{\end{lemma}}
\newcommand{\ecor}{\end{corr}}
\newcommand{\edeff}{\end{deff}}

\newcommand{\eprop}{\end{proposition}}

\newcommand{\la}{\lambda}

\newcommand{\e}{\varepsilon}

\newcommand{\supp}{\text{supp }}
\renewcommand{\Pi}{\varPi}

\renewcommand{\epsilon}{\varepsilon}

\newcommand{\R}{{\mathbb R}}

\begin{document}

\title[Strichartz estimates for the Schr\"odinger equation on Zoll manifolds]
{Strichartz estimates for the Schr\"odinger equation on Zoll manifolds.}
\thanks{The first author was supported in part by the Simons Foundation. The second author was  supported in part by the NSF (DMS-2348996).}
\author{Xiaoqi Huang}
\author{Christopher D. Sogge}
\address{XH: Department of Mathematics, Lousiana State University, Baton Rouge, LA 70803}
\address{CDS: Department of Mathematics,  Johns Hopkins University,
Baltimore, MD 21218}

\begin{abstract}
We obtain optimal space-time estimates  in $L^q_{t,x}$ spaces for all $q\ge 2$ for solutions to the Schrödinger equation on Zoll manifolds, including, in particular, the standard round sphere $S^d$. The proof relies on the arithmetic properties of the
spectrum of the Laplacian on Zoll manifolds, as well as bilinear oscillatory integral estimates, which allow us to relate the problem to Strichartz estimate on one-dimensional tori.
\end{abstract}

\maketitle

\newsection{Introduction
}

Let  $(M,g)$ be a compact Riemannian manifold of
dimension $d\ge2$, $\Delta_g$ denotes
the associated Laplace-Beltrami operator and
\begin{equation}\label{00.1}
u(x,t)=\bigl(e^{-it\Delta_g}f\bigr)(x)
\end{equation}
be the solution of the Schr\"odinger equation
on $M\times \R$,
\begin{equation}\label{00.2}
i \partial_tu(x,t)=\Delta_gu(x,t), \quad u(x,0)=f(x).
\end{equation}

Recall that the universal estimates
of Burq, G\'erard and Tzvetkov~\cite{bgtmanifold} says that if $u$
is the solution of the Schr\"odinger equation \eqref{00.2},
then one has the mixed-norm Strichartz estimates
\begin{equation}\label{00.3'}
\|u\|_{L^p_tL^q_x(M\times [0,1])}
\lesssim \|f\|_{H^{1/p}(M)}
\end{equation}
for all {\em admissible} pairs $(p,q)$.  By the latter
we mean, as in Keel and Tao~\cite{KT},
\begin{equation}\label{00.4}
d(\tfrac12-\tfrac1q)=\tfrac2p\, \, \,
\text{and } \, \, 2\le q\le \tfrac{2d}{d-2} \, \,
\text{if } \, d\ge 3, \, \, \, 
\text{or } \, 2\le q<\infty \, \, 
\text{if } \, \, d=2.
\end{equation}
Also, in \eqref{00.3'} the mixed norm is defined by 
\begin{equation}
    \|u\|_{L^p_tL^q_x(M\times [0,1])}=\bigr(\, \int_0^1 \, \|u(\, \cdot\, , t)\|_{L^q_x(M)}^p \, dt\,\bigr)^{1/p},
\end{equation}
and $H^\mu$ denotes the
standard Sobolev space
\begin{equation}\label{00.5'}
\|f\|_{H^\mu(M)}=
\bigl\| \, (I+P)^\mu f\, \bigr\|_{L^2(M)}, 
\quad \text{with } \, \, P=\sqrt{-\Delta_g},
\end{equation}
and ``$\lesssim$'' in \eqref{00.3'} and, in what
follows, denotes an inequality with an implicit,
but unstated, constant $C$ which can change at each occurrence. 

By taking the initial data $f$ to be  eigenfunctions of $P$, it is not hard to see that one can not replace the interval $[0,1]$ in \eqref{00.3'} by $[0, T]$ for an arbitrary large constant $T$, without the implicit constant depending on $T$. Also, for the endpoint Strichartz estimates where
$p=2$ and $q=\tfrac{2d}{d-2}$ with $d\ge3$, the $\frac12$ derivative loss in the right side of \eqref{00.3'} is sharp on $S^d$  by letting $f$ to be the zonal eigenfunctions. See, e.g., \cite{blair2023strichartz,bgtmanifold} for more details. 

In this paper, we consider the analog of \eqref{00.3'} in the case $p=q$, and obtain the following sharp estimate on all Zoll manifolds, i.e., manifolds all of whose geodesics are closed with a common minimal period.

\begin{theorem}\label{main}
Let $(M, g)$ be a smooth Zoll manifold of dimension $d\ge 2$. Then for all $q\ge 2$,
the solution $u$ of \eqref{00.2} satisfies 
\begin{equation}\label{Sn011}
\|u\|_{L^q_{t,x}(M\times [0,1] )}\le C \|f\|_{H^s(M)}, \,\,\,s>\mu(q),
\end{equation}
where $\mu(q)=\max\{\frac{d-1}{2}(\frac12-\frac1q), \frac{d-1}{2}-\frac{d}{q}, \frac d2-\frac{d+2}q\}$.
\end{theorem}

Note that when $d=2$, \eqref{Sn011} splits into estimates for two ranges of $q$:
\begin{equation}\label{d2est}
\|u\|_{L^q_{t,x}(M\times [0,1])}\lesssim_\e
\begin{cases}
\|f\|_{H^{s_{\mathrm {sm}}+\e}(M)},  \quad s_{\mathrm {sm}}=\tfrac12(\tfrac12-\tfrac1q), \, \, q\in (2,\tfrac{14}3 ]
\\
\|f\|_{H^{s_{\mathrm {Sob}}+\e}(M)},  \quad s_{\mathrm {Sob}}=1-\tfrac4q,  \, \, q\in [\tfrac{14}3, \infty],
\end{cases}
\end{equation}
while, in higher dimensions $d\ge3$, there are three ranges of $q$:
\begin{equation}\label{d3est}
\|u\|_{L^q_{t,x}(M\times [0,1])}\lesssim_\e
\begin{cases}
\|f\|_{H^{s_{\mathrm {sm}}+\e}(M)},  \quad s_{\mathrm {sm}}=\tfrac{d-1}2(\tfrac12-\tfrac1q), \, \, q\in (2,\tfrac{2(d+1)}{d-1} ]
\\
\|f\|_{H^{s_{\mathrm {lg}}+\e}(M)},  \quad s_{\mathrm {lg}}=\tfrac{d-1}2-\tfrac{d}q, \, \, q\in  [\tfrac{2(d+1)}{d-1},4 ]
\\
\|f\|_{H^{s_{\mathrm {Sob}}+\e}(M)},  \quad s_{\mathrm {Sob}}=\frac d2-\frac{d+2}q,  \, \, q\in [4, \infty].
\end{cases}
\end{equation}

As we shall see, the estimates involving $s_{\mathrm{sm}}$ and $s_{\mathrm{lg}}$ are saturated by eigenfunctions,
with these two exponents being the ones
occurring in the eigenfunction $L^q$-estimates
in   \cite{sogge86} and \cite{sogge881} involving relatively small and large exponents $q$.  The estimates involving
$s_{\mathrm {Sob}}$, as we shall see, are optimal due 
to standard  functions saturating dyadic $L^2_x \to L^q_x$ Sobolev
estimates.

Let us also briefly mention the reason that there are only two estimates when $d=2$ and must be three for $d\ge3$.  We first note
that, when $d=2$, we have $s_{\mathrm{sm} }\ge s_{\mathrm {Sob}} \iff q\in (2,\tfrac{14}3]$.  Also,
$s_{\mathrm{lg}}\ge s_{\mathrm{Sob}}\iff q\le 4$ and $s_{\mathrm{lg}}\ge s_{\mathrm {sm}}\iff q\ge 6$.  Since $4\in (2,\tfrac{14}3]$, we
conclude that when $d=2$ we simply have $\max(s_{\mathrm{sm}}, s_{\mathrm{lg}}, s_{\mathrm{Sob}})=\max(s_{\mathrm{sm}},s_{\mathrm{Sob}})$.
On the other hand, in higher dimensions $d\ge3$, we have $4\notin (2,\tfrac{2(d+1)}{d-1}]$, and
$s_{\mathrm{Sob}}\ge  \max(s_{\mathrm{sm}}, s_{\mathrm{lg}})\iff q\in [4,\infty]$ and also $s_{\mathrm{sm}}\ge s_{\mathrm{lg}}\iff q\in (2,\tfrac{2(d+1)}{d-1}]$.
This accounts for why, unlike for the $d=2$ case, there must be three ranges of exponents in higher dimensions.  We also should point out (see e.g. \cite{sogge86})
that, on $S^d$ zonal eigenfunctions, $Z_\la$, have $L^q$-norms which are $\approx \la^{s_{\mathrm{lg}}}$ for $q\in [\tfrac{2(d+1)}{d-1},\infty]$, and thus,
if we take $f=Z_\la$ we immediately see that the estimates in \eqref{d3est} must be sharp (up to the arbitrary $\e>0$) for 
$q\in [\tfrac{2(d+1)}{d-1},4]$, while we note that such data do not saturate the $d=2$ bounds in \eqref{d2est} for any
exponent $q\in (2,\infty]$.

If $M=S^1$, \eqref{Sn011} holds without any derivative loss for $q=4$ due to a result of Zygmund \cite{zygmund1974fourier}. This was later generalized to $q=6$ with an arbitrary small loss of derivative by Bourgain~\cite{bourgain1993fourier}. For $d\ge 2$,
in the special case $q=4$, \eqref{Sn011} was established in \cite[Theorem 4]{bgtmanifold}. This result was generalized to $q>4$ for $d=3$ by Herr \cite{herr2013quintic}. In particular, for $q>4$ and $d=3$ it was shown in \cite{herr2013quintic} that  \eqref{Sn011} holds with $s=\frac d2-\frac{d+2}q$ using a result of Bourgain \cite{bourgain1989lambda} based on the circle method of Hardy and Littlewood. The method employed in \cite{herr2013quintic} also implies that \eqref{Sn011} hold with  $s=\frac d2-\frac{d+2}q$ for all $q>4$, $d\ge 3$, and $q\ge 6$  when $d=2$. Therefore,  the main new contribution in Theorem~\ref{main} concerns the case of relatively small  exponents $q$.  In the next section, we will describe the difficulties in proving \eqref{Sn011} for small exponents $q$, particularly when $d=2$ and $q\in (4,6)$. We do not address the critical case
where $s=\mu(q)$ in this paper. 

Note that  for the spheres $S^d$ with $d\ge 2$, if we replace the $H^s$ norm in \eqref{Sn011} with its $L^q$-based analog,
Chen, Duong, Lee and Yan \cite{chen2022sharp} proved the following estimate:
\begin{equation}\label{lpsn}
\|e^{it\Delta_g}f\|_{L^q_{t,x}(S^d\times [0,1] )}\le C \|(1+P)^sf\|_{L^{q}(S^d)}, \,\,\, s>\max\{0, \tfrac{d}{2}-\tfrac{d+2}{q}\},
\end{equation}
for all $q\ge 2$ if $d=2$, and $q\ge 4$ if $d\ge 3$. The authors also showed that the lower bound on $s$ is optimal, using an example similar to one we will discuss in the final section,  along with a  semiclassical dispersion estimate  
of Burq, G\'erard and Tzvetkov~\cite{bgtmanifold}.

The estimate in Theorem~\ref{main} can also be generalized to any compact space forms with positive curvature, whose spectrum is a subset of the eigenvalues of Laplacian on the sphere-for example, real projective spaces $RP^n$ and lens spaces. See also Zhang \cite{zhang2020strichartz} for related results on compact Lie groups.

In the case of flat tori, similar sharp estimate holds with $s>\mu(q)=\max\{\frac{d}{2}-\frac{d+2}{q},0\}$ by using the $\ell_2$ decoupling theorem of Bourgain-Demeter \cite{BoDe}, together with Sobolev estimates. See also Killip and Visan \cite{killip2014scale} for the critical case $s=\mu(q)$ when $q>\frac{2(d+2)}{d}$, as well as the reference therein for a  summary of prior work on Strichartz estimates on square and irrational tori.
We shall also mention that the same result holds in the Euclidean space $\R^d$ globally in time if $q\ge \frac{2(d+2)}{d}$ by using the Euclidean analog of \eqref{00.3'}. 

As in earlier works, the proof of Theorem~\ref{main} relies on the arithmetic properties of the spectrum of the Laplacian
on Zoll manifolds. More explicitly, if the geodesics of $M$ are of period $T$, then it is known that there exist an integer $\alpha\ge 0$ and $A>0$, such that the spectrum of $\sqrt{-\Delta_g}$ is contained in $ \cup _{k=0}^\infty I_k$ where 
\begin{equation}\label{spe}
   I_k=[\frac{2\pi}{T}(k+\frac{\alpha}4)-\frac Ak, \frac{2\pi}{T}(k+\frac{\alpha}4)+\frac Ak].
\end{equation}
 Here $\alpha$  is the so called Maslov index of the closed geodesics, see e.g., \cite{DuistermaatGuillemin,weinstein,zelditch1997fine}. As a special case, when $M$ is the standard round sphere $S^d$,  the eigenvalues of $\sqrt{-\Delta_g}$ are $\sqrt{k(k+d-1)}$ for integers $k\ge 0$. The proof also uses the 
spectral projection bounds of the second author~\cite{sogge881}, which gives rise to the first two exponents 
in the definition 
of  $\mu(q)$ in \eqref{Sn011}. 
Additionally,
the proof of Theorem~\ref{main} also relies on the bilinear oscillatory integral estimates of
H\"ormander~\cite{HormanderFLP} and
 Lee~\cite{LeeBilinear} which is related to 
 work of Tao, Vargas and Vega \cite{TaoVargasVega}.

The authors would like to thank the referees for their careful reading and valuable suggestions, which improved the exposition.

\newsection{Proof of Theorem~\ref{main}.}

Let us first introduce the Littlewood-Paley cutoff functions, which allow us to reduce matters to certain dyadic estimates.
More explicitly, let us fix a Littlewood-Paley bump function
$\beta$ satisfying
\begin{equation}\label{00.11}
\beta \in C^\infty_0((1/2,2)) \quad
\text{and } \, \, 1=\sum_{m=-\infty}^\infty 
\beta(2^{-m}s), \, \, s>0.
\end{equation}
Then, if we set $\beta_0(s)=1-\sum_{m=1}^\infty
\beta(2^{-m}s)\in C^\infty_0(\R_+)$ and
$\beta_m(s)=\beta(2^{-m}s)$, $m=1,2,\dots$, we
have (see e.g., \cite{SFIO2})
\begin{equation}\label{00.12}
\|h\|_{L^q(M)}\approx
\bigl\| \, (\, \sum_{m=0}^\infty |\beta_m(P)h|^2\, 
)^{1/2} \, \bigr\|_{L^q(M)}, \, \, \,
1<q<\infty.
\end{equation} 
Trivially, $$\|\beta_0(P)e^{-it\Delta_g}\|_{L^2_x(M)
\to L^p_tL^q_x(M\times [0,1] )}=O(1),\,\,\forall\,\, p, q\ge 2,$$
and, similarly, such results
where $m=0$ is replaced by a small fixed $m\in {\mathbb N}$ are
also standard.
So, as noted
in  Burq, G\'erard and Tzvetkov~\cite{bgtmanifold}, one can use \eqref{00.12} and
Minkowski's inequality to see that 
\eqref{Sn011}  follows from  
\begin{equation}\label{Sn01}
\|e^{-it\Delta_{g}}\beta(P/\la)f\|_{L^q_{t,x}(M\times [0,1])}\le C_\e \la^{\mu(q)+\e} \|f\|_{L^2(M)}, \,\,\,\forall \e>0,
\end{equation}
assuming $\la=2^m\gg 1$.

To prove \eqref{Sn01}, we begin with the following:
\begin{lemma}\label{mainL4}
Let $(M, g)$ be a smooth Zoll manifold of dimension $d\ge 2$, for any $2\le q\le \infty$, we have for $\la=2^m\gg1$
\begin{equation}\label{Sn}
\|e^{-it\Delta_{g}}\beta(P/\la)f\|_{L^{q}_xL^4_t([0, 1]\times M)}\le C_\e \la^{\sigma(q)+\e} \|f\|_{L^2(M)}, \,\,\,\forall \e>0,
\end{equation}
where $\sigma(q)=\max\{\frac{d-1}{2}(\frac12-\frac1q), \frac{d-1}{2}-\frac{d}{q}\}$.
\end{lemma}

\eqref{Sn} is a generalization of Theorem 4 in \cite{bgtmanifold}, where the special case $q=4$ was handled. As in \cite{bgtmanifold}, the proof of \eqref{Sn} relies on the  arithmetic properties of the eigenvalues of Laplacian on Zoll manifolds.


\begin{proof}
To prove \eqref{Sn}, it suffices to show that whenever we fix $\rho\in \mathcal{S}(\R)$ satisfying $\supp\hat\rho\subset (-\frac12, \frac12)$ and $\rho\ge 0$, we have 
\begin{equation}\label{Snsmooth}
\|\rho(t)e^{-it\Delta_{g}}\beta(P/\la)f\|_{L^{q}_xL^4_t(\R\times M)}\le C_\e \la^{\sigma(q)+\e} \|f\|_{L^2(M)}, \,\,\,\forall \e>0.
\end{equation}

By a dilation of the metric
we may assume that the geodesics of $M$ have $2\pi$ as a common period.  By \eqref{spe},  the spectrum of $\sqrt{-\Delta_g}$ is then contained in $ \cup _{k=1}^\infty I_k$ where 
\begin{equation}\label{spe1}
    I_k=[k+\frac{\alpha}4-\frac Ak,  k+\frac{\alpha}4+\frac Ak],
\end{equation} for some constant $A>0$ and integer $\alpha \ge 0$.
Now let ${P}_k$ denote the projection operator such that ${P}_k f=f $ if Spec ${f \in I_k}$.
 Since we are assuming $f=\beta(P/\la)f$ with $\la=2^m\gg 1$, by \eqref{00.11}, it suffices to consider $k\gg 1$. Note that for $k\gg 1$, 
 the intervals
 $I_k$ are disjoint, 
 and so we have 
\begin{equation}\label{2plus2}
    \begin{aligned}
       \Big( \rho(t)&e^{-it\Delta_{g}}\beta(P/\la)f\Big)^2\\
       =&\sum_{k,l=\la/4}^{4\la}\big( \rho(t)e^{-it\Delta_{g}}\beta(P/\la)P_k f \big)\cdot \big(\rho(t)e^{-it\Delta_{g}}\beta(P/\la)P_\ell f\big)  \\
       =& \sum_{k,l=\la/4}^{4\la} \tilde P_k f \tilde P_\ell f,
    \end{aligned}
    \end{equation}
    if
\begin{equation}\label{2pluss}
       \tilde P_k f=\rho(t)e^{-it\Delta_{g}}\beta(P/\la)P_k f .
    \end{equation}
 By using the spectral projection estimates of the second author \cite{sogge881}, we have 
 \begin{equation}\label{projectionn}
   \|  P_k f\|_{L^q_x}\le C k^{\sigma(q)}\|f\|_{L^2_x}.
\end{equation}
 Thus
 it is not hard to see that for any fixed $t$,
\begin{equation}\label{projection}
   \| \tilde P_k f\|_{L^q_x}\le C\rho(t) k^{\sigma(q)}\|f\|_{L^2_x}.
\end{equation}
Also, by \eqref{spe1} and the Fourier support property of $\rho$,  for each fixed $k, \ell$, there exist some uniform constant $C_0$ such that the $t$-Fourier transform  of 
$\tilde P_k f \tilde P_\ell f$ is supported in 
\begin{equation}\label{tfourier}
    [(k+\frac{\alpha}4)^2+(\ell+\frac{\alpha}4)^2-C_0, (k+\frac{\alpha}4)^2+(\ell+\frac{\alpha}4)^2+C_0],
\end{equation}

Next, let us fix
\begin{equation}\label{22.16}
\eta\in C^\infty_0((-1,1)) \quad \text{satisfying } \, \, 1\equiv \sum_{j=-\infty}^\infty \eta(\tau-j),
\end{equation}
and define $\eta_j(D_t)=\eta(D_t-j)$, which is essentially Fourier restriction to 
the interval $[j-1, j+1]$. By Bernstein's inequality, it is not hard to see that
\begin{equation}\label{Bersteineta}
   \|\eta_j(D_t)\|_{L^p\rightarrow L^q}\lesssim 1,\,\,\,1\le p\le q\le \infty.
\end{equation}

By \eqref{2plus2} and Plancherel's theorem, we have 
\begin{equation}
\begin{aligned}
    \|\rho(t)e^{-it\Delta_{g}}\beta(P/\la)f\|^2_{L^{q}_xL^4_t(\R\times M)}&= 
   \big \|\sum_{k,l=\la/4}^{4\la} \tilde P_k f \tilde P_\ell f \big\|_{L^{q/2}_xL^2_t} \\
   &\lesssim \big \|\sum_{(k,\ell)\in I_j} \eta_j(D_t)\tilde P_k f \tilde P_\ell f \big\|_{L^{q/2}_x\ell_j^2L^2_t},
\end{aligned}
\end{equation}
where 
\begin{multline}
   (k,\ell)\in I_j, \, \,\,\text{if}\,\, \la/4\le k, \ell\le  4\la \\
   \text{and}\,\,\,(k+\frac{\alpha}4)^2+(\ell+\frac{\alpha}4)^2 \in [j-C_0-1, j+C_0+1].
\end{multline}
It is straightforward to check that each fixed pair $(k,\ell)$ is only contained in a set $I_j$ for finitely many $j$, and $\eta_j(D_t)\tilde P_k f \tilde P_\ell f $ is nonzero only when $j\in [\la^2/32, 32\la^2]$ and $ (k,\ell)\in I_j$. Also, it follows from classical number theory  that
we have the cardinality estimate
 $\# I_j \lesssim_\e \la^\varepsilon,\,\forall \varepsilon>0$ (see e.g., \cite{rankin1987grosswald}).

By \eqref{Bersteineta} and Minkowski inequality, we have 
\begin{equation}
\begin{aligned}
 \big \|\sum_{(k,\ell)\in I_j} \eta_j(D_t)\tilde P_k f \tilde P_\ell f \big\|_{L^{q/2}_x\ell_j^2L^2_t} \lesssim &  \big \|\sum_{(k,\ell)\in I_j} \tilde P_k f \tilde P_\ell f \big\|_{L^{q/2}_x\ell_j^2L^1_t}\\
 \lesssim &  \big \|\sum_{(k,\ell)\in I_j} \tilde P_k f \tilde P_\ell f \big\|_{L^1_tL^{q/2}_x\ell_j^2}.
\end{aligned}
\end{equation}
Note that 
\begin{align*}
 \big \|\sum_{(k,\ell)\in I_j} \tilde P_k f \tilde P_\ell f \big\|_{\ell_j^2}
&=\Big( \sum_{j\in [\la^2/32, 32\la^2]} \big|\sum_{(k,\ell)\in I_j}\tilde P_k f \tilde P_\ell f\big|^2 \Big)^{\frac12}\\
&\lesssim \Big( \sum_{j\in [\la^2/32, 32\la^2]}\# I_j\cdot\sum_{(k,\ell)\in I_j} \big| \tilde P_k f \tilde P_\ell f\big|^2\Big)^{\frac12} \\
&\lesssim_\e \la^\varepsilon \sum_{k\in [\la/10, 10\la]} \big|  \tilde P_k f\big|^2.
\end{align*}
Thus, by Minkowski's inequality and \eqref{projection}
\begin{align*}
\big \|\sum_{(k,\ell)\in I_j} \tilde P_k f \tilde P_\ell f \big\|_{L^1_tL^{q/2}_x\ell_j^2}
\le& C_\e \la^{\e} \big\| \big(\sum_{k\in [\la/10, 10\la]} \big|  \tilde P_k f\big|^2\big)^{1/2}\big\|^2_{L^2_tL^{q}_x} \\
\le & C_\e \la^{\e} \big(\sum_{k\in [\la/10, 10\la]} \|\tilde{P}_kf\|_{L^2_tL^{q}_x}^2\big) \\
\le & C_\e \la^{2\sigma(q)+\e} \|f\|^2_{L^2},
\end{align*}
which finishes the proof of the lemma.
\end{proof}

Now we shall see how we can apply Lemma~\ref{mainL4} to prove  \eqref{Sn01} except 
 for the cases where $d=2$ and $q\in (4,6)$.

To see this, it is natural to separately consider the following two cases.

\noindent(i) First, let us prove \eqref{Sn01} if
$d\ge 3$, $2\le q\le  \frac{2(d+1)}{d-1}$ or $d=2$, $2\le q\le 4$.
In this case, 
we see from \eqref{d2est} and \eqref{d3est} that $\mu(q)=\sigma(q)$.  Also, $q\le 4$.
Thus, if we integrate the $t$-variable first and use H\"older's inequality, it follows from Lemma~\ref{mainL4} that
\begin{multline}
    \|e^{-it\Delta_{g}}\beta(P/\la)f\|_{L^q_{t,x}(M\times [0,1]))}\\
    \le\|e^{-it\Delta_{g}}\beta(P/\la)f\|_{L^q_xL^4_t(M\times [0,1] )}\le C_\e \la^{\sigma(q)+\e} \|f\|_{L^2(M)}.
\end{multline}

\noindent(ii) The other case where we can use Lemma~\ref{mainL4} is when
$d\ge 2$, $q\ge \frac{2(d+1)}{d-1}$. In this case, 
$\mu(q)=s_{{\mathrm{Sob}}}=\frac{d}2-\frac{d+2}q$.  Consequently,
it suffices to prove that whenever we fix $\rho\in \mathcal{S}(\R)$
 satisfying $\supp\hat\rho\subset (-\frac12, \frac12)$, we have 
\begin{equation}\label{Snsmooth1}
\|\rho(t)e^{-it\Delta_{g}}\beta(P/\la)f\|_{L^{q}_{t,x}( M\times \R)}\le C_\e 
\la^{\frac{d}2-\frac{d+2}q+\e} \|f\|_{L^2(M)}, \,\,\,\forall \e>0.
\end{equation}
Let us fix
\begin{equation}\label{22.7}
\tilde \beta\in C^\infty_0((1/8,8)) \quad \text{satisfying } \, \,
\tilde \beta=1 \, \, 
\text{on } \, \, [1/6,6],
\end{equation}
then by 
Bernstein's inequality, we have 
\begin{equation}\label{Berstein1}
   \|\tilde\beta(D_t/\la^2)\|_{L^p\rightarrow L^q}\lesssim \la^{\frac{2}{p}-\frac{2}{q}},\,\,\,1\le p\le q\le \infty.
\end{equation}
Also, by \eqref{00.11} and the support property of $\rho$, we have 
\begin{equation}
    \rho(t)e^{-it\Delta_{g}}\beta(P/\la)f=\tilde\beta(D_t/\la^2)\rho(t)e^{-it\Delta_{g}}\beta(P/\la)f.
\end{equation}

Thus, if we use \eqref{Snsmooth} for $q>4$ and \eqref{Berstein1}, we have 
\begin{equation}\label{pb4}
\begin{aligned}
        \| \rho(t)e^{-it\Delta_{g}}\beta(P/\la)f\|_{L^q_{t, x}( M\times \R)} &\lesssim \la^{\frac12-\frac2q}  \| \rho(t)e^{-it\Delta_{g}}\beta(P/\la)f\|_{L^q_{x}L^4_t(\R\times M )}  \\
        &\lesssim   C_\e \la^{\frac12-\frac2q}  \la^{\sigma(q)+\e} \|f\|_{L^2(M)},
\end{aligned}
\end{equation}
as desired since $\frac12-\frac2q+\sigma(q)=\frac12-\frac2q+\frac{d-1}{2}-\frac{d}{q}=\frac d2-\frac {d+2}q$ when $q\ge \tfrac{2(d+1)}{d-1}$.

The remaining case when $d=2,\, 4<q<6$ requires considerably more work, 
since, for this range of exponents 
$q$ and dimension, we have 
$$\frac12-\frac2q+\sigma(q)> \max\bigl(\frac{d}2-\frac{d+2}q, \, \frac12(\frac12-\frac1q)\bigr)
=\max(s_{{\mathrm{Sob}}}, s_{{\mathrm{sm}}}).$$
Consequently, we cannot directly use  Sobolev estimates in the $t$ variable and  Lemma~\ref{mainL4}
to obtain \eqref{Sn011} in this case.

\newsection{Proof of Theorem~\ref{main} when $4<q<6$ and $d=2$.}


In this section we shall prove the remaining estimates in \eqref{Sn011} where $q\in (4,6)$ and $d=2$.  Note that
$14/3\in (4,6)$ and that $\mu(14/3)$ agrees with both $s_{\mathrm{sm}}$ and $s_{\mathrm{Sob}}$ when $q=14/3$.  
Consequently, by interpolation and \eqref{d2est}, we see that these remaining cases of \eqref{Sn011} would follow
from showing that, for the critical exponent for $d=2$ of $q=14/3$, we have
\begin{equation}\label{aa3'}
\|e^{-it\Delta_{g}}\beta(P/\la)f\|_{L^{\frac{14}{3}}_{t,x}(M\times [0,1])}\le C_\e \la^{\frac{1}{7}+\e} \|f\|_{L^2(M)}, \,\,\,\forall \e>0,
\end{equation}
when $M=M^2$ is a Zoll surface.
 Indeed, by Littlewood-Paley estimates as in the previous section, \eqref{a3'} implies \eqref{Sn011} when $q=\frac{14}{3}$,
  which yields the desired bounds for the other exponents in $q\in (4,6)$ by interpolation with the previously obtained 
  bounds for $q=4$ and $6$.

For simplicity, let us first prove \eqref{a3'} when $M=S^2$. 
In the end of the section, we shall describe the modifications needed to prove \eqref{Sn011} for all Zoll surfaces.

Recall that 
\begin{equation}\label{a4}
e^{-it\Delta_{g}}\beta(P/\la)f=\sum_{k\in \mathbb{N}}e^{itk(k+1)}\beta(\sqrt{k(k+1)}/\la)H_kf
\end{equation}
where $H_k$ denotes the projection operator onto the eigenspace of $P=\sqrt{-\Delta_g}$ with eigenvalue $\sqrt{k(k+1)}$. 

For present and future use, let us choose a  bump function now satisfying
\begin{equation}\label{a9}
\beta_0\in C^\infty_0((1-\delta_1\delta,1+\delta_1\delta)),   \, \, \, 
\beta_0(\tau)=1 \, \, \, \text{if } \, \tau\in (1-\delta_1\delta/2,1+\delta_1\delta/2).
\end{equation}
Since the interval (1/2, 2) can be covered by finitely many intervals of length $\approx \delta_1\delta$, to prove \eqref{aa3'}, it suffices to show
\begin{equation}\label{a3'}
\|e^{-it\Delta_{g}}\beta_0(P/\la)f\|_{L^{\frac{14}{3}}_{t,x}(M\times [0,1])}\le C_\e \la^{\frac{1}{7}+\e} \|f\|_{L^2(M)}, \,\,\,\forall \la\gg 1, \e>0.
\end{equation}
Also, in order to exploit calculations involving the half-wave operators we shall define smoothed out spectral projection operators
of the form
\begin{equation}\label{a5}
\rho_k =\rho(k-P),\quad P=\sqrt{-\Delta_g},
\end{equation}
where 
\begin{equation}\label{a6}
\rho\in {\mathcal S}(\R), \, \, \rho(0)=1, \, \, \,\hat\rho\ge 0,\,\,\,
\text{and } \, \, \text{supp }\Hat \rho\subset \delta\cdot [1-\delta_0,1+\delta_0]
=[\delta-\delta_0\delta, \delta+\delta_0\delta].
\end{equation}
The size of the small fixed positive constants $\delta, \delta_0, \delta_1$ will be specified later, and they are crucial for the bilinear oscillatory integral estimates that arise.

Let $c_k=(\rho(k-\sqrt{k(k+1)}))^{-1}$. 
Since $\rho(0)=1$ and $|k-\sqrt{k(k+1)}|\le \frac12$ for $k\approx\la\gg 1$, by choosing $\delta$ in \eqref{a6} sufficiently small, we have $|c_k|\lesssim 1$. Moreover, 
\begin{equation}\label{a7}
e^{-it\Delta_{g}}\beta(P/\la)f=\sum_{k\in \mathbb{N}}e^{it(k^2+k)}\beta_0(\sqrt{k(k+1)}/\la)\rho_k c_kH_kf.
\end{equation}

As in the earlier works \cite{huang2024curvature, blair2022improved}, in order to use the local harmonic analysis tools, 
it will be convenient to localize a bit more using microlocal cutoffs.  Specifically, let us write
\begin{equation}\label{a8}
I=\sum_{j=1}^N B_j(x,D)
\end{equation}
where each $B_j\in S^0_{1,0}(M)$ is a zero order pseudo-differential operator with symbol supported in a small neighborhood
of some $(x_j,\xi_j)\in S^*M$.  The size of the support will be described shortly; however, we point out now that these operators
will not depend on the spectral parameter $\la \gg1$.

If $\tilde \beta_0\in C^\infty_0((0,\infty))$ equals one in a neighborhood of the support of the bump function
$\beta_0$ in  \eqref{a9},
then the dyadic operators
\begin{equation}\label{a11}
B=B_{j,\la}=B_j\circ \tilde\beta(P/\la)
\end{equation}
are uniformly bounded on $L^p(M)$, i.e.,
\begin{equation}\label{a12}
\|B\|_{p\to p}=O(1) \quad \text{for } \, \, 1\le p\le \infty.
\end{equation}

Since the number of $B_j$ operators is finite, to prove \eqref{a3'} when $M=S^2$, it suffices to show for $\tilde \rho_k=B\circ \rho_k$ and $\beta_{0,k,\la}=\beta_0(\sqrt{k(k+1)}/\la)$
\begin{equation}\label{a13}
\|\sum_{k\in \mathbb{N}}e^{it(k^2+k)}\beta_{0,k,\la}\tilde\rho_kH_kf\|_{L^{\frac{14}{3}}_{t,x}(S^{2}\times [0,1])}\le C_\e \la^{\frac{1}{7}+\e} \|f\|_{L^2(S^2)}, \,\,\,\forall \e>0.
\end{equation}
In fact, for any $f$, one can define  $\tilde f$ such that $H_k \tilde f=c_k H_k f$ if $k\approx \la$ and $H_k \tilde f=0$ otherwise. Since $|c_k|\lesssim 1$, it is straightforward to check that $\|\tilde f\|_2\lesssim \|f\|_2$.  Therefore, applying \eqref{a13} to $\tilde f$ yields \eqref{a3'}.

We shall also need a microlocal decomposition as in our recent work \cite{huang2024curvature}, which allows us to use the bilinear 
harmonic analysis techniques as in \cite{LeeBilinear} and
\cite{TaoVargasVega}.
The decomposition is similar to the one  in \cite[\S 2.2]{huang2024curvature} with only minor modifications. For completeness, we include here the details of the construction in  dimension 2 below. 

First recall that the symbol $B(x,\xi)$ of $B$ in \eqref{a8} is supported in a small
conic neighborhood of some $(x_0,\xi_0)\in S^*M$.  We may assume that its symbol has
small enough support so that we may work in a coordinate chart $\Omega$ and that
$x_0=0$, $\xi_0=(0,1)$ and $g_{jk}(0)=\delta^j_k$ in the local coordinates.
So, we shall assume that $B(x,\xi)=0$ when $x$ is outside a small relatively compact neighborhood
of the origin or $\xi$ is outside of a small conic neighborhood of $(0,1)$.

Next, let us define the microlocal cutoffs that we shall use.   We fix a function
$a\in C^\infty_0({\mathbb R}^{2})$ supported in $\{z: \, |z_k|\le 1, \, \, 1\le k\le 2\}$
 which satisfies
\begin{equation}\label{m1}
\sum_{j\in {\mathbb Z}^{2}}a(z-j)\equiv 1.
\end{equation}
We shall use this function to build our microlocal cutoffs.
By the above, we shall focus on defining them 
 for $(y,\eta)\in S^*\Omega$ with    $y$ near the origin
 and  $\eta$ in a small conic neighborhood of $(0,1)$. 
We shall let
$$\Pi=\{y: \, y_{2}=0\}$$
be the points in $\Omega$ whose last coordinate vanishes. 
 For $y\in \Pi$ near $0$ and $\eta$ near $(0, 1)$ we can
just use the functions $a(\theta^{-1}(y_1,\eta_1)-j)$, $j\in {\mathbb Z}^{2}$ to obtain cutoffs of scale $\theta$.  We will always have
$\theta\in[\la^{-1/2+\e_0}, 1]$ for some $\e_0>0$ that will be specified later.

We can then extend the definition to a neighborhood of $(0,(0, 1))$ by setting for $(x,\xi)\in S^*\Omega$ in this neighborhood
\begin{equation}\label{m2}
a^\theta_j(x,\xi)=a(\theta^{-1}(y_1,\eta_1)-j) \quad
\text{if } \, \, \Phi_s(x,\xi)=(y_1,0,\eta_1,\eta_{2}) \, \, \, \text{for some} \,\,s.
\end{equation}
Here $\Phi_s$ denotes geodesic flow in $S^*\Omega$.  Thus, $a^\theta_j(x,\xi)$ is constant on all geodesics
$(x(s),\xi(s))\in S^*\Omega$ with $x(0)\in \Pi$ near $0$ and $\xi(0)$ near $(0, 1)$.   As a result,
\begin{equation}\label{m3}
a^\theta_j(\Phi_s(x,\xi))=a^\theta_j(x,\xi)
\end{equation}
for $s$ near $0$ and $(x,\xi)\in S^*\Omega$ near $(0,(0, 1))$.

We then extend the definition of the cutoffs to a conic neighborhood of $(0,(0, 1))$  in $T^*\Omega \, \backslash \, 0$ by setting
\begin{equation}\label{m4}
a^\theta_j(x,\xi)=a^\theta_j(x,\xi/p(x,\xi)),
\end{equation}
where $p(x,\xi)=|\xi|_{g(x)}$ is the principal symbol of $P=\sqrt{-\Delta_g}$.

Notice that if $((y_1)_\nu,(\eta_1)_\nu)=\theta j=\nu$ and $\gamma_\nu$ is the geodesic in $S^*\Omega$ passing through $((y_1)_\nu,0,\eta_\nu)\in S^*\Omega$
with $\eta_\nu\in S^*_{((y_1)_\nu, 0)}\Omega$ having $(\eta_1)_\nu$ as its first  coordinate then
\begin{equation}\label{m5}
a^\theta_j(x,\xi)=0 \quad \text{if } \, \, \,
\text{dist }\bigl((x,\xi), \gamma_\nu\bigr)\ge C_0\theta,
\, \, \nu=\theta j,
\end{equation}
for some fixed constant $C_0>0$.  


Finally, if $\psi \in C^\infty_0(\Omega)$ equals
one in a neighborhood of the $x$-support of $B(x,\xi)$,
and if $\bar \beta\in C^\infty_0((0,\infty))$ equals one in a neighborhood of the support of the bump function
in \eqref{a11} we define
\begin{equation}\label{qnusymbol}
A_\nu^\theta(x,\xi)=\psi(x) \, a_j^\theta(x,\xi) \, \bar\beta\bigl(p(x,\xi)/\la\bigr),
\quad \nu =\theta j\in \theta \cdot {\mathbb Z}^{2}.
\end{equation}
By \eqref{m2}, \eqref{m4} and the assumption that $\theta\in [\la^{-1/2+\e_0},1]$, it is not hard to check that the symbol $A^\theta_\nu(x,\xi)$ satisfies 
\begin{equation}\label{m6}
\bigl|\partial_x^\sigma \partial_\xi^\gamma A^\theta_\nu(x,\xi)\bigr| \lesssim \langle \xi \rangle^{(\frac12-\e_0)|\sigma|-(\frac12+\e_0)|\gamma|}.
\end{equation}
Hence the pseudo-differential operators
$A_\nu^\theta(x,D)$ with these symbols belong to a bounded subset
of $S^0_{\frac12+\e_0, \frac12-\e_0}(M)$. This implies
\begin{align}\label{a.33a}
\|A^{\theta}_\nu h\|_{ L^2(M)} \lesssim \|h\|_{L^2(M)}.
\end{align}
Also for later use, it is not hard to show the following almost orthogonal type inequality:
\begin{align}\label{a.33}
\sum_\nu\|A^{\theta_0}_\nu h\|^2_{ L^2(M)} &\lesssim \|h\|_{L^2(M)}.
\end{align}
Equation \eqref{a.33a} follows from \eqref{m6} and the $L^2$ boundedness for of pseudo-differential operators  with symbols in the H\"ormander classes, 
while \eqref{a.33} follows additionally from the bounded overlap of the symbols of the operators $A^{\theta_0}_\nu$.

Next we note that by \eqref{m1}, \eqref{m2} and \eqref{qnusymbol}, we have that, as operators between any $L^p(M)\to L^q(M)$ spaces,
$1\le p,q\le \infty$, for $\theta\ge \la^{-1/2+\e}$ and $k\in \supp \beta(\cdot/\la)$ 
\begin{equation}\label{m11}
\tilde \rho_k =\sum_\nu \tilde \rho_k A^\theta_\nu +O(\la^{-N}), \, \forall \, N.
\end{equation}
This just follows from the fact that $R(x,D)=I-\sum_\nu A^\theta_\nu $ has symbol
supported outside of a neighborhood of $B(x,\xi)$, if, as we may, we assume that the latter is small.

In view of \eqref{m11} we have for $\theta_0=\la^{-1/2+\e_0}$
\begin{equation}\label{m13}
\bigl(\tilde \rho_k h\bigr)^2=\sum_{\nu_1,\nu_2} \bigl(\tilde \rho_k A^{\theta_0}_{\nu_1} h\bigr) \,
\bigl(\tilde \rho_k A^{\theta_0}_{\nu_2} h\bigr) \, + \, O(\la^{-N}\|h\|_2^2).
\end{equation}


If
$\theta_0=\la^{-1/2+\e_0}$ then the $\nu=\theta_0\cdot  {\mathbb Z}^{2}$
 index a $\la^{-1/2+\e_0}$-separated set in
${\mathbb R}^{2}$.  We need to organize the pairs of indices $\nu_1,\nu_2$ in \eqref{m13} as in many earlier works
(see \cite{LeeBilinear} and \cite{TaoVargasVega}).  We consider dyadic cubes $\tau^\theta_\mu$ in 
${\mathbb R}^{2}$ of side length $\theta=2^m\theta_0$, $m=0,1,\dots$, with
$\tau^\theta_\mu$ denoting translations of the cube $[0,\theta)^{2}$ by
$\mu=\theta{\mathbb Z}^{2}$.  Then two such dyadic cubes of side length $\theta$ are said to be
{\em close} if they are not adjacent but have adjacent parents of side length $2\theta$, and, in that case, we write
$\tau^\theta_{\mu_1} \sim \tau^\theta_{\mu_2}$.  Note that close cubes satisfy $\text{dist }(\tau^\theta_{\mu_1},\tau^\theta_{\mu_2})
\approx \theta$ and so each fixed cube has $O(1)$ cubes which are ``close'' to it.  Moreover, as noted in \cite{TaoVargasVega},
any distinct points $\nu_1,\nu_2\in {\mathbb R}^{2}$ must lie in a unique pair of close cubes in this Whitney decomposition
of ${\mathbb R}^{2}$.  Consequently, there must be a unique triple $(\theta=\theta_02^m, \mu_1,\mu_2)$ such that
$(\nu_1,\nu_2)\in \tau^\theta_{\mu_1}\times \tau^\theta_{\mu_2}$ and $\tau^\theta_{\mu_1}\sim \tau^\theta_{\mu_2}$.  We remark that by choosing $B$
to have small support we need only consider $\theta=2^m\theta_0\ll 1$.

Taking these observations into account implies that that the bilinear sum in \eqref{m13} can be organized as follows:
\begin{multline}\label{m14}
\sum_{\{m\in {\mathbb N}: \, m\ge 10 \, \, \text{and } \, 
\theta=2^m\theta_0\ll 1\}}
\sum_{\{(\mu_1, \mu_2): \, \tau^\theta_{\mu_1}
\sim \tau^\theta_{ \mu_2}\}}
\sum_{\{(\nu_1, \nu_2)\in
\tau^\theta_{\mu_1}\times \tau^\theta_{ \mu_2}\}}
\bigl(\tilde \rho_{k_1}
A^{\theta_0}_{\nu_1} h\bigr) 
\cdot \bigl(\tilde \rho_{k_2}
A^{\theta_0}_{ \nu_2} h\bigr)
\\
+\sum_{(\nu_1, \nu_2)\in \Xi_{\theta_0}} 
\bigl( \tilde \rho_{k_1} A^{\theta_0}_{\nu_1} \bigr) 
\cdot \bigl( \tilde \rho_{k_2}
A^{\theta_0}_{ \nu_2} 
h\bigr)
,
\end{multline}
where $\Xi_{\theta_0}$ indexes the remaining pairs such
that $|\nu_1- \nu_2|\lesssim \theta_0=\la^{-1/2+\e_0}$,
including the diagonal ones where $\nu_1= \nu_2$.

Note that by \eqref{m13}, if we define 
\begin{equation}\label{5.2a}
\begin{aligned}
P_{k,\nu}f=e^{it(k^2+k)}\beta_{0,k,\la}\tilde\rho_{k}A^{\theta_0}_{\nu} f 
\end{aligned}
\end{equation}
we have 
\begin{equation}\label{5.2}
\begin{aligned}
    (\sum_{k\in \mathbb{N}}e^{it(k^2+k)}\beta_{0,k,\la}\tilde\rho_kH_kf)^2 
    =\sum_{k_1,k_2} \sum_{\nu_1,\nu_2} P_{k_1,\nu_1}(H_{k_1}f) \cdot P_{k_2,\nu_2}(H_{k_2}f)+ O(\la^{-N}\|f\|_2^2).
\end{aligned}
\end{equation}

Let us define
\begin{equation}\label{diag}
\Upsilon_0(f)=\sum_{k_1,k_2}\sum_{(\nu_1,\nu_2)\in \Xi_{\theta_0}} P_{k_1,\nu_1}(H_{k_1}f) P_{k_2,\nu_2}(H_{k_2}f),
\end{equation}
and
\begin{equation}\label{5.4}
\Upsilon_j(f)=\sum_{k_1,k_2}
\sum_{\{(\mu_1, \mu_2): \, \tau^\theta_{\mu_1}
\sim \tau^\theta_{ \mu_2},\, \theta=2^j\}}
\sum_{\{(\nu_1, \nu_2)\in
\tau^\theta_{\mu_1}\times \tau^\theta_{ \mu_2}\}}
 P_{k_1,\nu_1}(H_{k_1}f)\cdot P_{k_2,\nu_2}(H_{k_2}f).
\end{equation}
  Then, by \eqref{m14}, we have
\begin{equation}\label{5.5}
 (\sum_{k\in \mathbb{N}}e^{it(k^2+k)}\beta_{0,k,\la}\tilde\rho_kH_kf)^2 = \Upsilon_0(f)+\sum_{j: \,2^{10}\theta_0\le2^j
\ll 1}\Upsilon_j(f)+ O(\la^{-N}\|f\|_2^2),
\end{equation}
with the last term containing the error terms in \eqref{5.2}.

 Using ideas from Bourgain's\cite{bourgain1993fourier} proof of Strichartz estimates on the one-dimensional torus, we shall be able to obtain the following
estimates:

\begin{lemma}\label{diag1}  If $\Upsilon_0(f)$
is defined as in \eqref{diag}, and, as above $\theta_0=\la^{-1/2+\e_0}$, then for
all $\e>0$ we have
\begin{equation}\label{star}
\|\Upsilon_0(f)\|_{L^{3}_{t,x}(S^{2}\times [0,1])}
\lesssim_\e \la^{\frac1{3}+\e}\la^{\frac{2\e_0}3}\|f\|^2_{L^2(S^2)}.
\end{equation}
Similarly, for $\Upsilon_j(f)$ as in \eqref{5.4} we have
\begin{equation}\label{star1}
\|\Upsilon_j(f)\|_{L^{3}_{t,x}(S^{2}\times [0,1])}
\lesssim_\e \la^{\frac2{3}+\e}2^{\frac{2j}3}\|f\|^2_{L^2(S^2)}.
\end{equation}
\end{lemma}

We also require the following lemma which is a consequence of the bilinear estimates proved by the authors in \cite{huang2024curvature} using
bilinear oscillatory integral estimates of Lee~\cite{LeeBilinear} and slightly simplified 
variants of the arguments in \cite{BlairSoggeRefined}, \cite{blair2015refined} and 
\cite{SBLog}.

\begin{lemma}\label{leelemma}  
 If $\Upsilon_j(f)$ is defined as in \eqref{5.4}, and, as above $\theta_0\ll 2^j\ll 1$ then for
all $\e>0$ we have
\begin{equation}\label{5.7}
\|\Upsilon_j(f)\|_{L^{2}_{t,x}(S^{2}\times [0,1])}
\lesssim_\e \la^{\e}2^{-j/2}\|f\|^2_{L^2(S^2)},
\end{equation}
assuming the conic support of $B(x,\xi)$ in \eqref{a11} as well as $\delta$, $\delta_0$ and $\delta_1$
in \eqref{a9} and \eqref{a6}
are sufficiently small.
\end{lemma}
Before giving the proofs of the lemmas, let us show how we can use them to prove \eqref{a13}.
Note that $\frac{3}{7}=\frac12\cdot\frac47+\frac13\cdot\frac37$. By \eqref{star1}, \eqref{5.7} and H\"older's inequality, we have 
\begin{equation}\label{rj}
\|\Upsilon_j(f)\|_{L^{\frac73}_{t,x}}
\lesssim_\e \la^{\frac2{7}+\e}\|f\|_2^2.
\end{equation}

On the other hand,
by \eqref{5.5}, we have 
\begin{equation}\label{starl}
\|\Upsilon_0(f)\|_{L^{2}_{t,x}}
\lesssim \sum_j\|\Upsilon_j(f)\|_{L^{2}_{t,x}}+\left\|\sum_{k\in \mathbb{N}}e^{itk^2}\beta_{0,k,\la}\tilde\rho_kH_kf\right\|^2_{L^{4}_{t,x}}.
\end{equation}
By \eqref{a12} and \eqref{Sn01} for $q=4$, we have 
\begin{equation}\label{starll}
\left\|\sum_{k\in \mathbb{N}}e^{it(k^2+k)}\beta_{0,k,\la}\tilde\rho_kH_kf\right\|^2_{L^{4}_{t,x}}\lesssim_\e \la^{\frac14+\e}\|f\|_2^2.
\end{equation}
This together with \eqref{5.7} implies 
\begin{equation}\label{starlll}
\|\Upsilon_0(f)\|_{L^{2}_{t,x}}
\lesssim_\e \la^{\frac14+\e}\|f\|_2^2.
\end{equation}
By \eqref{star}, \eqref{starlll} and H\"older's inequality, we have 
\begin{equation}\label{farq}
\|\Upsilon_0(f)\|_{L^{\frac73}_{t,x}}
\lesssim_\e \la^{\frac2{7}(1+\e_0)+\e}\|f\|_2^2.
\end{equation}

Thus, if we choose $\e_0$ sufficiently small such that $\e_0\ll\e$ and apply \eqref{5.5} once more, we obtain
\begin{equation}\label{main2}
\begin{aligned}
\left\|\sum_{k\in \mathbb{N}}e^{it(k^2+k)}\beta(k/\la)\tilde\rho_kH_kf\right\|_{L^{\frac{14}3}_{t,x}}&\lesssim \|\Upsilon_0(f)\|^{\frac12}_{L^{\frac73}_{t,x}}+\left(\sum_j\|\Upsilon_j(f)\|_{L^{\frac73}_{t,x}}\right)^{\frac12}+O(\la^{-N})\|f\|_2\\
&\lesssim_\e 
(\log\la)^{\frac12}\la^{\frac1{7}+2\e}\|f\|_2\lesssim_\e \la^{\frac1{7}+3\e}\|f\|_2,
\end{aligned}
\end{equation}
as desired.

Thus, it remains to prove Lemma~\ref{diag1} and Lemma~\ref{leelemma}.
\begin{proof}[Proof of Lemma~\ref{diag1}]
We first present the proof of \eqref{star1}.
To prove \eqref{star1},  recall that by \eqref{m11}, we have for a given $\theta=2^j\ge \theta_0$, 
\begin{equation}\label{5.16a}
\tilde \rho_k A^{\theta_0}_\nu h
=\sum_{\tilde \mu \in \theta\cdot {\mathbb Z}^{2}} 
\tilde \rho_k A^{\theta}_{\tilde \mu}A^{\theta_0}_\nu h +O(\la^{-N}\|h\|_2).
\end{equation}
Thus, for a given 
pair of dyadic cubes $\tau^\theta_{\mu_1}, \tau^\theta_{\mu_2}$ with $\tau^\theta_{\mu_1}\sim \tau^\theta_{\mu_2}$
\begin{multline}\label{5.17a}
\sum_{k_1,k_2}
\sum_{\{(\nu_1, \nu_2)\in
\tau^\theta_{\mu_1}\times \tau^\theta_{ \mu_2}\}}
 P_{k_1,\nu_1}(H_{k_1}f)\cdot P_{k_2,\nu_2}(H_{k_1}f)
\\
=\sum_{k_1,k_2} \sum_{\substack{\tau^{\theta}_{\tilde \mu_1}\cap \overline{\tau}^\theta_{\mu_1}
\ne \emptyset \\ \tau^{\theta}_{\tilde \mu_2}\cap \overline{\tau}^\theta_{\mu_2}\ne \emptyset}}
 P_{k_1, \tilde \mu_1}(\sum_{\nu_1\in \tau^\theta_{\mu_1}} A^{\theta_0}_{\nu_1}H_{k_1}f) \cdot P_{k_2, \tilde \mu_2}(\sum_{\nu_2\in \tau^\theta_{\mu_2}}A^{\theta_0}_{\nu_2}H_{k_2}f)
+O(\la^{-N} \|f\|_2^2),
\end{multline}
where
\begin{equation}\label{munu}
\begin{aligned}
P_{k,\mu}f=e^{it(k^2+k)}\beta_{0,k,\la}\tilde\rho_{k} A^{\theta}_{ \mu}f .
\end{aligned}
\end{equation}
Here
$\overline{\tau}^\theta_{\mu_1}$ and $\overline{\tau}^\theta_{\mu_2}$ are cubes with the same centers but 10 times the side length
of $\tau^\theta_{\mu_1}$ and $\tau^\theta_{\mu_2}$, respectively.
We obtain \eqref{5.17a} from the fact that the product of the symbol of $A^{\theta}_{\tilde \mu}$ and $A^{\theta_0}_\nu$
vanishes if $\tau^{\theta}_{\tilde \mu}\cap \overline{\tau}^\theta_\mu = \emptyset$ and $\nu\in \tau^\theta_\mu$.
 Note also that for the fixed pair $\tau_{\mu_1}^\theta\sim \tau^\theta_{\mu_2}$ of $\theta$-cubes there are only $O(1)$
summands involving $\tilde \mu_1$ and $\tilde \mu_2$ in the right side of \eqref{5.17a}.

Based on this we claim that we would have \eqref{star1} if we could prove the following key result
\begin{proposition}\label{diagmain}
If $P_{k,  \mu}$ is defined as in \eqref{munu} with $ \mu \in \theta\cdot {\mathbb Z}^{2}$ and  $\theta=2^j$, then for all $j$ such that $\theta_0\le2^j\ll1$, we have 
\begin{equation}\label{dmain}
        \left\|\sum_{k}P_{k,  \mu}f_k\right\|_{L^{6}_{t,x}(S^{2}\times [0,1])}\lesssim \la^{\frac1{3}+\e}2^{\frac{j}3}(\sum_k\|f_k\|^2_{L^2(S^2)})^{\frac12}
\end{equation}
\end{proposition}
Before giving the proof of \eqref{dmain}, let us verify the claim.
We first note that if 
$h_1=\sum_{\nu_1\in \tau^\theta_{\mu_1}} A^{\theta_0}_{\nu_1} h$
and 
$h_2=\sum_{\nu_2\in \tau^\theta_{\mu_2}} A^{\theta_0}_{\nu_2} h,
$
then by  the almost orthogonality of the $A^{\theta_0}_\nu$ operators, 
$$\|h_1\|_2^2 \lesssim \sum_{\nu_1 \in \tau^\theta_{\mu_1}} \|A^{\theta_0}_{\nu_1} h\|_2^2
\quad \text{and } \, \, 
\|h_2\|_2^2 \lesssim \sum_{\nu_2 \in \tau^\theta_{\mu_2}} \|A^{\theta_0}_{\nu_2} h\|_2^2.
$$

Thus, by \eqref{5.17a}, \eqref{dmain} and 
Minkowski’s inequality, we obtain, up to the negligible  $O(\la^{-N})\|f\|^2_2$ errors, that
\begin{equation}\nonumber
\begin{aligned}
    &\|\Upsilon_j(f)\|_{L^{3}_{t,x}}
    \\
    &\le  \sum_{(\mu_1,\mu_2): \tau^\theta_{\mu_1} \sim \tau^\theta_{\mu_2}} \sum_{\substack{\tau^{\theta}_{\tilde \mu_1}\cap \overline{\tau}^\theta_{\mu_1}
\ne \emptyset \\ \tau^{\theta}_{\tilde \mu_2}\cap \overline{\tau}^\theta_{\mu_2}\ne \emptyset}}\left\|\sum_{k_1,k_2} P_{k_1, \tilde \mu_1}(\sum_{\nu_1\in \tau^\theta_{\mu_1}} A^{\theta_0}_{\nu_1}H_{k_1}f)  P_{k_2, \tilde \mu_2}(\sum_{\nu_2\in \tau^\theta_{\mu_2}} A^{\theta_0}_{\nu_2}H_{k_2}f)\right\|_{L^{3}_{t,x}} \\
     &\le \sum_{(\mu_1,\mu_2): \tau^\theta_{\mu_1} \sim \tau^\theta_{\mu_2}} \sum_{\substack{\tau^{\theta}_{\tilde \mu_1}\cap \overline{\tau}^\theta_{\mu_1}
\ne \emptyset \\ \tau^{\theta}_{\tilde \mu_2}\cap \overline{\tau}^\theta_{\mu_2}\ne \emptyset}}\left\|\sum_{k_1}P_{k_1, \tilde \mu_1}(\sum_{\nu_1\in \tau^\theta_{\mu_1}} A^{\theta_0}_{\nu_1}H_{k_1}f)  \right\|_{L^{6}_{t,x}} \cdot \left\|\sum_{k_2}P_{k_2, \tilde \mu_2}(\sum_{\nu_2\in \tau^\theta_{\mu_2}} A^{\theta_0}_{\nu_2}H_{k_2}f)\right\|_{L^{6}_{t,x}}\\
&\lesssim_\e  \la^{\frac2{3}+\e}2^{\frac{2j}3}   \sum_{(\mu_1,\mu_2): \tau^\theta_{\mu_1} \sim \tau^\theta_{\mu_2}}
\bigl(\sum_{\nu_1\in \tau^\theta_{\mu_1}}\sum_{k_1}\|A_{\nu_1}^{\theta_0}H_{k_1}f\|_2^2\bigr)^{1/2}
\bigl(\sum_{\nu_2\in \tau^\theta_{\mu_2}}\sum_{k_2}\|A_{\nu_2}^{\theta_0}H_{k_2}f\|_2^2\bigr)^{1/2} 
\\
&\lesssim  \la^{\frac2{3}+\e}2^{\frac{2j}3} \left(\sum_{\mu_1} \sum_{\nu_1\in \tau^\theta_{\mu_1}}\sum_{k_1}\|A_{\nu_1}^{\theta_0}H_{k_1}f\|_2^2\right)^{\frac12}\left(\sum_{\mu_2} \sum_{\nu_2\in \tau^\theta_{\mu_2}}\sum_{k_2}\|A_{\nu_2}^{\theta_0}H_{k_2}f\|_2^2\right)^{\frac12}
\\
&\lesssim  \la^{\frac2{3}+\e}2^{\frac{2j}3}  \|f\|^2_2. 
\end{aligned}
\end{equation}
In the above we used the fact that for each $\tau^\theta_\mu$ there are $O(1)$ cubes $\tau^{\theta}_{\tilde \mu}$ with
$\tau^{\theta}_{\tilde \mu}\cap \overline{\tau}^\theta_\mu\ne \emptyset$ and for each $\tau^\theta_{\mu_1}$ there are $O(1)$ $\tau^\theta_{\mu_2}$ with
$\tau^\theta_{\mu_1} \sim \tau^\theta_{\mu_2}$, and we also used \eqref{a.33}.

This completes the proof of \eqref{star1}. The proof of \eqref{star} follows similarly. To see this, note that by applying \eqref{dmain} when $\theta=\theta_0$ and using  Minkowski’s inequality
\begin{equation}\label{d1}
\begin{aligned}
    \|\Upsilon_0(f)\|_{L^{3}_{t,x}}
    &\le  \sum_{(\nu_1,\nu_2)\in \Xi_{\theta_0}} \left\|\sum_{k_1,k_2}P_{k_1,\nu_1}(H_{k_1}f) P_{k_2,\nu_2}(H_{k_2}f) \right\|_{L^{3}_{t,x}} \\
     &\le  \sum_{(\nu_1,\nu_2)\in \Xi_{\theta_0}} \left\|\sum_{k_1}P_{k_1,\nu_1}(H_{k_1}f)  \right\|_{L^{6}_{t,x}} \cdot \left\|\sum_{k_2}P_{k_2,\nu_2}(H_{k_2}f) \right\|_{L^{6}_{t,x}} \\
          &\le  C\sum_\nu \left\|\sum_{k}P_{k,\nu}(H_{k}f) \right\|^2_{L^{6}_{t,x}}
\end{aligned}
\end{equation}
In the above we used the fact that for each $\nu_1$ there are $O(1)$ values of $\nu_2$ such that $(\nu_1,\nu_2)\in \Xi_{\theta_0}$. Note that   by \eqref{m11} and the fact that the symbols of $A^{\theta_0}_\nu$ operators have bounded overlap, we have 
\begin{equation}\nonumber
\tilde \rho_k A^{\theta_0}_\nu h
=\sum_{|\nu'-\nu|\lesssim \theta_0} 
\tilde \rho_k A^{\theta_0}_{\nu'}A^{\theta_0}_\nu h +O(\la^{-N}\|h\|_2).
\end{equation}
Therefore, by \eqref{dmain} and \eqref{a.33}
\begin{equation}\nonumber
\begin{aligned}
 \sum_\nu  \left\|\sum_{k}P_{k,\nu}(H_{k}f) \right\|^2_{L^{6}_{t,x}}&\le \sum_\nu  \left\|\sum_{k,\nu'}P_{k,\nu'}(A^{\theta_0}_\nu H_{k}f) \right\|^2_{L^{6}_{t,x}}+O(\la^{-N}\|f\|_2)\\
&\lesssim\la^{\frac1{3}+\e}\la^{\frac{2\e_0}3}\sum_\nu\sum_k\left\|A^{\theta_0}_{\nu} H_{k}f\right\|^2_{2}++O(\la^{-N}\|f\|_2)\\
&\lesssim\la^{\frac1{3}+\e}\la^{\frac{2\e_0}3}\|f\|^2_{2}. 
\end{aligned}
\end{equation}

From this we conclude that, in order to conclude the proof of Lemma~\ref{diag1}, we just need to prove Proposition~\ref{diagmain}.

\noindent\textbf{Proof of Proposition~\ref{diagmain}.}

To prove \eqref{dmain}, we shall follow an idea from Bourgain's~\cite{bourgain1993fourier} proof of Strichartz estimates for the Schr\"odinger equation on $\mathbb{T}$. First, write
\begin{equation}\label{d4}
    \begin{aligned}
\left\|\sum_{k}P_{k,  \mu}f_k\right\|^3_{L^{6}_{t,x}}
=\left\|\sum_{k_1,k_2,k_3}P_{k_1,\mu}f_{k_1}P_{k_2,\mu}f_{k_2}P_{k_3,\mu}f_{k_3}\right\|_{L^{2}_{t,x}}.
    \end{aligned}
\end{equation}
The main step is to prove
\begin{equation}\label{d5}
    \begin{aligned}
&\left\|\sum_{k_1,k_2,k_3}P_{k_1,\mu}f_{k_1}P_{k_2,\mu}f_{k_2}P_{k_3,\mu}f_{k_3}\right\|^2_{L^{2}_{t,x}}\\
&\lesssim\sum_{\ell_1, \ell_2\in \mathbb{Z}}\left\|\sum_{(k_1,k_2,k_3)\in S_{\ell_1,\ell_2}}P_{k_1,\mu}f_{k_1}P_{k_2,\mu}f_{k_2}P_{k_3, \mu}f_{k_3}\right\|^2_{L^{2}_{t,x}}+ O(\la^{-N} (\sum_k\|f_k\|^2_{2})^{3}),
    \end{aligned}
\end{equation}
where 
\begin{multline}\label{d5a}
    S_{\ell_1,\ell_2}=\{(k_1,k_2,k_3): k_1^2+k_1+k_2^2+k_2+k_3^2+k_3=\ell_1, \\ \,\,k_1+k_2+k_3\in [\ell_2\la 2^{2j}, (\ell_2+1)\la2^{2j}\}.
\end{multline}

To prove this, write
\begin{equation}\label{d6}
    \begin{aligned}
&\left\|\sum_{k_1,k_2,k_3}P_{k_1,\mu}f_{k_1}P_{k_2,\mu}f_{k_2}P_{k_3,\mu}f_{k_3}\right\|^2_{L^{2}_{t,x}}\\
&=\sum_{k_1,\cdots,k_6}\int \int P_{k_1,\mu}f_{k_1}P_{k_2,\mu}f_{k_2}P_{k_3,\mu}f_{k_3} \overline{P_{k_4,\mu}f_{k_4}P_{k_5,\mu}f_{k_5}P_{k_6,\mu}f_{k_6}} dtdx.
    \end{aligned}
\end{equation}
By taking the $dt$ integral, it is not hard to see that the above integral is nonzero only if $k_1^2+k_1+k_2^2+k_2+k_3^2+k_3=k_4^2+k_4+k_5^2+k_5+k_6^2+k_6$, which naturally gives rise to the first constraint in \eqref{d5a}. Thus to prove \eqref{d5}, it suffices to show 
\begin{multline}\label{d7}
 \int P_{k_1,\mu}f_{k_1}P_{k_2,\mu}f_{k_2}P_{k_3,\mu}f_{k_3} \overline{P_{k_4,\mu}f_{k_4}P_{k_5,\mu}f_{k_5}P_{k_6,\mu}f_{k_6}} dx=O(\la^{-N} \|f_{k_1}\|_{2}\cdots\|f_{k_6}\|_{2})\\
 \text{if}\,\,\, \left|k_1+k_2+k_3-k_4-k_5-k_6\right|\ge C\la2^{2j}.
\end{multline}

To see this,  let us first collect some facts about the kernels of the operators $\tilde \rho_k A^{\theta}_{\mu}$ in \eqref{munu} that we shall use.
As we shall shortly see they are highly concentrated near a specific geodesic in $S^2$.  Recall that $A^{\theta}_{\mu}(x,D)$ is a 
``directional operator'' with $\mu\in \theta\cdot {\mathbb Z}^{2}$ and, by \eqref{m5}, symbol $A^{\theta}_\mu(x,\xi)$ highly
concentrated near a unit speed geodesic
\begin{equation}\label{5.22}
\gamma_\mu(s)=(x_\mu(s),\xi_\mu(s))\in S^*\Omega.
\end{equation}
Since $\gamma_\mu$ is of unit speed, we have $d_g(x_\mu(s),x_\mu(s'))=|s-s'|$.

To state the properties of the kernels $K^{\theta}_{k,\mu}(x,y)$ of the operators $\tilde \rho_k A^{\theta}_{\mu}$, as in 
earlier works, it is convenient to work in Fermi normal coordinates about the spatial geodesic $\overline{\gamma}_\mu =\{x_\mu(s)\}$.
In these coordinates, the geodesic becomes part of the last coordinate axis, i.e., $(0, s)$ in $\R^2$, with, as in the earlier
construction of the symbols of the $A^{\theta}_\mu$, $s$ being close to $0$.  For the remainder of this section we shall let
$x=(x_1, x_2)$ denote these Fermi normal coordinates about our geodesic $\overline{\gamma}_\mu$ associated with
$A^{\theta}_\mu$.  We then have
\begin{equation}\label{5.23}
d_g((0, x_2),(0, y_2))=|x_2-y_2|,
\end{equation}
and, moreover, on $\overline{\gamma}_\mu$ we have that the metric is just $g_{jk}(x)=\delta^k_j$ if $x=(0, x_2)$, and, 
additionally, all of the Christoffel symbols vanish there as well.


We need to following lemma which is a consequence of Lemma~A.4 in \cite{huang2024curvature}.
\begin{lemma}\label{loclemma}  
Fix $0<\delta\ll \tfrac12 \text{Inj }M$ and let $K^{\theta}_{k,\mu}$ be the kernel of $\tilde\rho_{k}A^{\theta}_{\mu}$,  we have for $k\approx \la$ and $\theta=2^{j}$ with $\la^{-\frac12+\e_0}\le2^j\ll1$
\begin{equation}\label{k2}
K^{\theta}_{k,\mu}(x,y)
=k^{\frac{1}2} e^{ik d_g(x,y)}
a_\mu(k; x,y) +O(\la^{-N})
\end{equation}
where
\begin{equation}\label{k3}
\bigl| \, \bigl(\tfrac\partial{\partial x_2})^{m_1}
\bigl(\tfrac\partial{\partial y_2})^{m_2} 
D^\beta_{x,y}a_\mu\, \bigr|
\le C_{m_1,m_2,\beta} \, 2^{-|\beta|j}.
\end{equation}
Furthermore, for small $\theta$  there is a 
constant $C_0$ so that the above $O(\la^{-N})$ errors
can be chosen so that the amplitudes have the
following support properties:
\begin{equation}\label{4'} a_\mu(k;x,y)=0 \, \, \,
\text{if } \, \, |x_1|
+ |y_1|\ge C_02^{j},
\end{equation}
as well as, for small $\delta, \delta_0>0$ as in
\eqref{a6}
\begin{equation}\label{k5}
a_\mu(k;x,y)=0 \, \, \,
\text{if } \, \, |d_g(x,y)-\delta|\ge 2\delta_0\delta,  
 \, \, \text{or } \, \, x_2-y_2<0.
\end{equation}
\end{lemma}
Note that in \cite{huang2024curvature}, the above lemma was proved for all $\theta\ge \la^{-\frac18}$. However, the same arguments apply for all $\theta\ge \la^{-\frac12+\e_0}$ for any fixed $\e_0>0$.  The stationary phase arguments used require $\e_0$ to be positive.

Now let us describe some properties of the phase function
\begin{equation}\label{5.31} 
\varphi(x,y)=d_g(x,y)
\end{equation}
of our kernels in \eqref{k2}.  First, in addition
to \eqref{5.23}, since we are working in the 
above Fermi normal coordinates we have
\begin{equation}\label{5.32}
\partial\varphi/\partial x_1, \,
\partial\varphi/\partial y_1=0, \, \,
\text{if } \, \, x_1=y_1=0.
\end{equation}
Also, note that by \eqref{4'} and \eqref{k5}, if 
we assume $2^j\le \delta_0\delta/10$, then whenever the 
amplitude is nonzero, we have $x_2-y_2\approx \delta$. Consequently,
\begin{equation}\label{5.33}
\tilde \varphi(x,y)=\varphi(x,y)-(x_2-y_2)
\end{equation}
vanishes to second order when $x_1=y_1=0$ and the 
amplitude is nonzero.  This means 
\begin{equation}\label{5.34}
D^\beta_{x_2,y_2} 
\tilde \varphi(x_1,x_2, y_1,y_2)
=O_\beta(2^{2j}) \, \, \,\,
\text{if } \, \, |x_1|, |y_1|=O(2^j).
\end{equation}

Now we are ready to prove \eqref{d7}. By \eqref{k2}
\begin{equation}\label{d8}
    \begin{aligned}
      &\int  P_{k_1,\mu}f_{k_1}P_{k_2,\mu}f_{k_2}P_{k_3,\mu}f_{k_3} \overline{P_{k_4,\mu}f_{k_4}P_{k_5,\mu}f_{k_5}P_{k_6,\mu}f_{k_6}} dx\\
 &=c\int_{y_{1},\cdots, y_{6}\in \R^2}\int_{x\in \R^2}  e^{i(k_1d_g(x,y_1)+k_2d_g(x,y_2)+k_3d_g(x,y_3)-k_4d_g(x,y_4)-k_5d_g(x,y_5)-k_6d_g(x,y_6)) }\\
 &\qquad\cdot a_\mu(k_1; x,y_1)f_{k_1}(y_1)a_\mu(k_2; x,y_2)f_{k_2}(y_2)a_\mu(k_3; x,y_3)f_{k_3}(y_3)\\
 &\qquad\cdot\overline{a_\mu(k_4; x,y_4)f_{k_4}(y_4)a_\mu(k_5; x,y_5)f_{k_5}(y_5)a_\mu(k_6; x,y_6)f_{k_6}(y_6)} dx dy_1\cdots dy_6 \\
 &\qquad+O(\la^{-N} \|f_{k_1}\|_{2}\cdots\|f_{k_6}\|_{2}), 
    \end{aligned}
\end{equation}
where the constant $c\approx \la^3$ depends on $t, k_1,\dots, k_6$ but is independent of $x$.
By \eqref{5.33}, \eqref{5.34} and the fact that $|x_1|, |y_1|\lesssim 2^j$ due to \eqref{4'}, we obtain, for $x=(x_1,x_2)$,
\begin{multline*}
    |\partial_{x_2}(k_1d_g(x,y_1)+k_2d_g(x,y_2)+k_3d_g(x,y_3)-k_4d_g(x,y_4)-k_5d_g(x,y_5)\\-k_6d_g(x,y_6))|
    \gtrsim \la2^{2j},\,\,\,
   \text{if}\,\,\, \left|k_1+k_2+k_3-k_4-k_5-k_6\right|\ge C\la2^{2j}. 
\end{multline*}
Moreover, \eqref{5.34} also implies
\begin{multline*}
    |\partial^n_{x_2}(k_1d_g(x,y_1)+k_2d_g(x,y_2)+k_3d_g(x,y_3)-k_4d_g(x,y_4)-\\k_5d_g(x,y_5)-k_6d_g(x,y_6))|
    \lesssim \la2^{2j} \,\,\,\,\forall\,n\ge 2.
\end{multline*}
Combining these bounds with \eqref{k3} and integrating by parts in $x_2$ yields \eqref{d7}.

Now we shall finish the proof of Proposition \ref{diagmain}. by \eqref{d5},
\begin{equation}\label{d9}
    \begin{aligned}
&\left\|\sum_{k}P_{k,\mu}f_k\right\|_{L^{6}_{t,x}}
=\left\|\sum_{k_1,k_2,k_3}P_{k_1,\mu}f_{k_1}P_{k_2,\mu}f_{k_2}P_{k_3,\mu}f_{k_3}\right\|^{\frac13}_{L^{2}_{t,x}}\\
&\quad\lesssim
\left(\sum_{\ell_1, \ell_2\in \mathbb{Z}}\left\|\sum_{(k_1,k_2,k_3)\in S_{\ell_1,\ell_2}}P_{k_1,\mu}f_{k_1}P_{k_2,\mu}f_{k_2}P_{k_3,\mu}f_{k_3}\right\|^2_{L^{2}_{t,x}}\right)^{\frac16}+  O(\la^{-N} (\sum_k\|f_k\|^2_{2})^{\frac16}).
    \end{aligned}
\end{equation}

Note that if $(k_1,k_2,k_3)\in S_{\ell_1,\ell_2}$,  $P_{k_1,\mu}f_{k_1}P_{k_2,\mu}f_{k_2}P_{k_3,\mu}f_{k_3}$ is nonzero only if $\ell_1\approx \la^2$. Moreover, for any fixed $\ell_1, \ell_2$ with $\ell_1\approx \la^2$, the number of choices of integers $k$ such that $k_1+k_2+k_3=k$ is $\la2^j$. For such a fixed $k$, it follows from classical number theory (see e.g., \cite{bourgain1993fourier}, (2.40)-(2.42)) that 
$$\#\{(k_1,k_2): k_1^2+k_1+k_2^2+k_2+(k-k_1-k_2)^2+k-k_1-k_2=\ell_1\}\lesssim_\e \la^\e.
$$
Consequently, $\# S_{\ell_1,\ell_2} \lesssim_\e \la^{1+\e}2^{2j}$. Thus, by  the Cauchy-Schwarz inequality,
\begin{equation}\label{d10}
    \begin{aligned}
&\left(\sum_{\ell_1, \ell_2\in \mathbb{Z}}\left\|\sum_{(k_1,k_2,k_3)\in S_{\ell_1,\ell_2}}P_{k_1,\mu}f_{k_1}P_{k_2,\mu}f_{k_2}P_{k_3,\mu}f_{k_3}\right\|^2_{L^{2}_{t,x}}\right)^{\frac16}\\
&\lesssim_\e \la^{\frac{1+\e}6}2^{j/3}\left(\sum_{\ell_1, \ell_2\in \mathbb{Z}}\sum_{(k_1,k_2,k_3)\in S_{\ell_1,\ell_2}}\left\|P_{k_1,\mu}f_{k_1}P_{k_2,\mu}f_{k_2}P_{k_3,\mu}f_{k_3}\right\|^2_{L^{2}_{t,x}}\right)^{\frac16}\\
&\lesssim_\e \la^{\frac{1+\e}6}2^{j/3}\left(\sum_{\ell_1, \ell_2\in \mathbb{Z}}\sum_{(k_1,k_2,k_3)\in S_{\ell_1,\ell_2}}\left\|P_{k_1,\mu}f_{k_1}\right\|^2_{L^{6}_{t,x}}\left\|P_{k_2,\mu}f_{k_2}\right\|^2_{L^{6}_{t,x}}\left\|P_{k_3,\mu}f_{k_3}\right\|^2_{L^{6}_{t,x}}\right)^{\frac16}\\
&= \la^{\frac{1+\e}6}2^{j/3}\left(\sum_k\left\|P_{k,\mu}f_{k}\right\|^2_{L^{6}_{t,x}}\right)^{\frac12}.
    \end{aligned}
\end{equation}
Note that for fixed $k\approx \la$, it follows from the classical result of Sogge~\cite{sogge881} that 
\begin{equation}
   \|\tilde \rho_kf\|_{L_x^6} \lesssim \la^{\frac16} \|f\|_{L^2_x}
\end{equation}
Thus the right side of \eqref{d10} is bounded by 
\begin{equation}\label{d11}
    \begin{aligned}
\la^{\frac{1+\e}6}2^{j/3}\left(\sum_k\left\|P_{k,\mu}f_{k}\right\|^2_{L^{6}_{t,x}}\right)^{\frac12}&\lesssim_\e  \la^{\frac13+\e}2^{j/3}\left(\sum_k\left\|A^{\theta}_{\mu} H_{k}f\right\|^2_{L^{2}_{x}}\right)^{\frac12} \\&\lesssim_\e  \la^{\frac13+\e}2^{j/3}(\sum_k\|f_k\|^2_{L^2})^{\frac12}.
    \end{aligned}
\end{equation}
In the  last inequality, 
we used \eqref{a.33a}.

By combining \eqref{d9}, \eqref{d10} and \eqref{d11}, we obtain \eqref{dmain}, which completes the proof of Proposition~\ref{diagmain}.
\end{proof}

\begin{proof}[Proof of Lemma~\ref{leelemma}]
    To prove \eqref{5.7}, we first claim that, it suffices to prove for any fixed $k_1,k_2\in \supp \beta(\, \cdot \, /\la)$ 
    \begin{equation}\label{a21}
\|T_{k_1,k_2}(h_1,h_2)\|_{L^{2}_{x}}
\lesssim_\e \la^{\e}2^{-j/2}\|h_1\|_2\|h_2\|_2,
\end{equation}
 where 
$$ T_{k_1,k_2}(h_1,h_2)=\sum_{\{(\mu_1, \mu_2): \, \tau^\theta_{\mu_1}
\sim \tau^\theta_{ \mu_2},\, \theta=2^j\}}
\sum_{\{(\nu_1, \nu_2)\in
\tau^\theta_{\mu_1}\times \tau^\theta_{ \mu_2}\}}\tilde\rho_{k_1}A^{\theta_0}_{\nu_1} h_1\tilde\rho_{k_2}A^{\theta_0}_{\nu_2} h_2.
$$

To verify the claim, note that by $L^2$ orthogonality in the $t$ variable, we have
\begin{equation}\label{a22}
    \begin{aligned}
\|&\Upsilon_j(f)\|^2_{L^{2}_{t,x}}\\
&=\sum_{\ell}\left\|\sum_{k_1^2+k_1+k_2^2+k_2=\ell}\beta(k_1/\la)\beta(k_2/\la)T_{k_1,k_2}(H_{k_1}f,H_{k_2}f)\right\|^2_{L^{2}_{x}}\\
&\lesssim_\e \sum_{\ell\approx \la^2}\sum_{k_1^2+k_1+k_2^2+k_2=\ell}\la^{\e}\left\|T_{k_1,k_2}(H_{k_1}f,H_{k_2}f)\right\|^2_{L^{2}_{x}}.
\end{aligned}
\end{equation}
In the second line above we used the fact that for fixed $\ell\approx \la^{2}$, the number of integer pairs $(k_1, k_2)$ satisfying $k_1^2+k_1+k_2^2+k_2=\ell$ is bounded above by $C_\e\la^{\e}$ for arbitrarily small $\e$.

By \eqref{a21},  the right side of \eqref{a22} is bounded by 
\begin{equation}\label{a23}
    \begin{aligned}
 \sum_{\ell\approx \la^2}\sum_{k_1^2+k_1+k_2^2+k_2=\ell}&\la^{\e}\left\|T_{k_1,k_2}(H_{k_1}f,H_{k_2}f)\right\|^2_{L^{2}_{x}}\\
 &\lesssim_\e  \sum_{\ell\approx \la^2}\sum_{k_1^2+k_1+k_2^2+k_2=\ell}\la^{\e}2^{-j}\|H_{k_1}f\|^2_2\|H_{k_2}f\|^2_2
 \\
  &\lesssim_\e \la^{\e}2^{-j}\|f\|^4_2.
\end{aligned}
\end{equation}

The proof of \eqref{a21} mostly follows from the same strategy as in the proof of Lemma~A.2 in \cite{huang2024curvature}, 
but we include the details here for the sake of  completeness.
Recall \eqref{m14} and \eqref{5.4} and note that for a given $\theta=2^m\theta_0$, $m\ge10$, we have for
each fixed $c_0>0$
\begin{equation}\label{5.16}
\tilde \rho_k A^{\theta_0}_\nu h
=\sum_{\tilde \mu \in (c_0\theta)\cdot {\mathbb Z}^{2}} 
\tilde \rho_k A^{c_0\theta}_{\tilde \mu}A^{\theta_0}_\nu h +O(\la^{-N}\|h\|_2).
\end{equation}
We are only considering $m\ge10$ due to the organization of the sum in the left side of \eqref{m14}.  As in 
\cite{BlairSoggeRefined}, we shall choose $c_0=2^{-m_0}<1$ to be small enough to ensure that we have the
separation needed to apply bilinear oscillatory integral estimates.

Keeping this in mind fix $m\ge10$ in the first sum in \eqref{m14}.  We then have for a given $c_0$ as above and
pairs of dyadic cubes $\tau^\theta_{\mu_1}, \tau^\theta_{\mu_2}$ with $\tau^\theta_{\mu_1}\sim \tau^\theta_{\mu_2}$
\begin{multline}\label{5.17}
\sum_{({\nu_1},\nu_2)\in \tau^\theta_{\mu_1} \times \tau^\theta_{\mu_2}} (\tilde \rho_{k_1} A^{\theta_0}_{\nu_1} h) \, (\tilde \rho_{k_2} A^{\theta_0}_{\nu_2} h)
\\
=\sum_{(\nu_1,\nu_2)\in \tau^\theta_{\mu_1} \times \tau^\theta_{\mu_2}}  \sum_{\substack{\tau^{c_0\theta}_{\tilde \mu_1}\cap \overline{\tau}^\theta_{\mu_1}
\ne \emptyset \\ \tau^{c_0\theta}_{\tilde \mu_2}\cap \overline{\tau}^\theta_{\mu_2}\ne \emptyset}}
(\tilde \rho_{k_1} A^{c_0\theta}_{\tilde \mu_1}A^{\theta_0}_{\nu_1} h) \, (\tilde \rho_{k_2} A^{c_0\theta}_{\tilde \mu_2}A^{\theta_0}_{\nu_2} h)
+O(\la^{-N} \|h\|_2^2),
\end{multline}
if $\overline{\tau}^\theta_{\mu_1}$ and $\overline{\tau}^\theta_{\mu_2}$ are cubes with the same centers but 11/10 times the side length
of $\tau^\theta_{\mu_1}$ and $\tau^\theta_{\mu_2}$, respectively, so that we have
$\text{dist }(\overline{\tau}^\theta_{\mu_1}, \overline{\tau}^\theta_{\mu_2})\ge \theta/2$ when $\tau^\theta_{\mu_1} \sim \tau^\theta_{\mu_2}$.
We obtain \eqref{5.17} from the fact that the product of the symbol of $A^{c_0\theta}_{\tilde \mu}$ and $A^{\theta_0}_\nu$
vanishes if $\tau^{c_0\theta}_{\tilde \mu}\cap \overline{\tau}^\theta_\mu = \emptyset$ and $\nu\in \tau^\theta_\mu$ since
$\theta=2^m\theta_0$ with $m\ge 10$.  Also note that we then have for fixed $c_0=2^{-m_0}$ small enough
\begin{equation}\label{5.18}
\text{dist }(\tau^{c_0\theta}_{\tilde \mu_1}, \tau^{c_0\theta}_{\tilde \mu_2})
\in [4^{-1}\theta, 4^2\theta], \, \, \text{if }  \, \tau_{\mu_1}^\theta\sim \tau^\theta_{\mu_2}, \, \,
\tau^{c_0\theta}_{\tilde \mu_1}\cap \overline{\tau}^\theta_{\mu_1} \ne \emptyset
\, \, \text{and } \, \tau^{c_0\theta}_{\tilde \mu_2}\cap \overline{\tau}^\theta_{\mu_2} \ne \emptyset.
\end{equation}
 Note also that if we fix $c_0$ then for our pair $\tau_{\mu_1}^\theta\sim \tau^\theta_{\mu_2}$ of $\theta$-cubes there are only $O(1)$
summands involving $\tilde \mu_1$ and $\tilde \mu_2$ in the right side of \eqref{5.17}.

Based on this we claim that we would have \eqref{a21} if we could prove the following key result, which is a consequence of Proposition~A.3 in \cite{huang2024curvature}.
\begin{proposition}\label{prop5.3}  Let $\theta=2^j$ with $2^{10}\theta_0\le 2^j\ll 1$.  Then we can fix
$c_0=2^{-m_0}$ small enough so that whenever
\begin{equation}\label{5.19} \text{dist }(\tau^{c_0\theta}_{\nu_1}, \tau^{c_0\theta}_{ \nu_2})\in [4^{-1}\theta, 4^2 \theta]
\end{equation}
one has the uniform bounds for each $\e>0$
\begin{equation}\label{5.20}
\int_{S^2} \bigl| (\tilde \rho_{k_1} A^{c_0\theta}_{\nu_1} h_1) \, (\tilde \rho_{k_2} A^{c_0\theta}_{\nu_2} h_2)\bigr|^2 \, dx
\le C_\e \la^{\e}2^{-j}\|h_1\|_{L^2(S^2)}^{2} \, 
\|h_2\|_{L^2(S^2)}^{2},
\end{equation}
\end{proposition}

The proof of this proposition is based on the bilinear oscillatory integral estimates of Lee~\cite{LeeBilinear}.
If $k_1=k_2$, \eqref{5.20} follows from the results in \cite[Proposition~A.3]{huang2024curvature}. If $k_1\neq k_2$, by choosing $\delta_1$ in \eqref{a9} sufficiently small, which may depend on the constant $c_0$,  so that $k_1,  k_2=\la(1+O(\delta_1\delta))$,   one can follow the same arguments as in the proof of \cite[Proposition A.3]{huang2024curvature}
or \cite[(3-6)]{BlairSoggeRefined} to get \eqref{5.20}. The role of fixed constants $k_1, k_2$ here is the same as the frozen variables $y_n, z_n$ in \cite{huang2024curvature}.  See the proof of Proposition A.3 in \cite{huang2024curvature} for more details.

Now
 let us verify the above claim.
We first note that if 
$h_1=\sum_{\nu_1\in \tau^\theta_{\mu_1}} A^{\theta_0}_{\nu_1} h$
and 
$h_2=\sum_{\nu_2\in \tau^\theta_{\mu_2}} A^{\theta_0}_{\nu_2} h,
$
then by  the almost orthogonality of the $A^{\theta_0}_\nu$ operators, 
$$\|h_1\|_2^2 \lesssim \sum_{\nu_1 \in \tau^\theta_{\mu_1}} \|A^{\theta_0}_{\nu_1} h\|_2^2
\quad \text{and } \, \, 
\|h_2\|_2^2 \lesssim \sum_{\nu_2 \in \tau^\theta_{\mu_2}} \|A^{\theta_0}_{\nu_2} h\|_2^2.
$$
Thus, \eqref{5.16}, \eqref{5.18}, \eqref{5.20} and Minkowski's inequality, we have 
\begin{multline}\label{5.21}
\bigl\| \sum_{(\mu_1,\mu_2): \tau^\theta_{\mu_1} \sim \tau^\theta_{\mu_2}}
\sum_{(\nu_1,\nu_2)\in \tau^\theta_{\mu_1} \times \tau^\theta_{\mu_2}} (\tilde \rho_{k_1} A^{\theta_0}_{\nu_1} h_1)
(\tilde \rho_{k_2} A^{\theta_0}_{\nu_2} h_2) \bigr\|_{L^{2}}
\\
\le \sum_{(\mu_1,\mu_2): \tau^\theta_{\mu_1} \sim \tau^\theta_{\mu_2}}
\bigl\|
\sum_{\substack{\tau^{c_0\theta}_{\tilde \mu_1}\cap \overline{\tau}^\theta_{\mu_1}
\ne \emptyset \\ \tau^{c_0\theta}_{\tilde \mu_2}\cap \overline{\tau}^\theta_{\mu_2}\ne \emptyset}}
(\tilde \rho_{k_1} A^{c_0\theta}_{\tilde \mu_1}(\sum_{\nu_1\in \tau^\theta_{\mu_1}}A_{\nu_1}^{\theta_0}h_1)) \cdot
(\tilde \rho_{k_2} A^{c_0\theta}_{\tilde \mu_2}(\sum_{\nu_2\in \tau^\theta_{\mu_2}}A_{\nu_2}^{\theta_0}h_2))\bigr\|_{L^{2}}
+O(\la^{-N}\|h_1\|_2\|h_2\|_2)
\\
\lesssim_\e \la^{\e}2^{-j/2} \sum_{(\mu_1,\mu_2): \tau^\theta_{\mu_1} \sim \tau^\theta_{\mu_2}}
\bigl(\sum_{\nu_1\in \tau^\theta_{\mu_1}}\|A_{\nu_1}^{\theta_0}h_1\|_2^2\bigr)^{1/2}
\bigl(\sum_{\nu_2\in \tau^\theta_{\mu_2}}\|A_{\nu_2}^{\theta_0}h_2\|_2^2\bigr)^{1/2} 
+O(\la^{-N}\|h_1\|_2\|h_2\|_2)
\\
\lesssim  \la^{\e}2^{-j/2}\left(\sum_\mu \sum_{\nu\in \tau^\theta_\mu}\|A_\nu^{\theta_0}h_1\|_2^2\right)^{\frac12}\left(\sum_\mu \sum_{\nu\in \tau^\theta_\mu}\|A_\nu^{\theta_0}h_2\|_2^2\right)^{\frac12}+O(\la^{-N}\|h_1\|_2\|h_2\|_2)
\\
\lesssim \la^{\e}2^{-j/2}\|h_1\|_2\|h_2\|_2 +O(\la^{-N}\|h_1\|_2\|h_2\|_2).
\end{multline}
In the above we used the fact that for each $\tau^\theta_\mu$ there are $O(1)$ cubes $\tau^{c_0\theta}_{\tilde \mu}$ with
$\tau^{c_0\theta}_{\tilde \mu}\cap \overline{\tau}^\theta_\mu\ne \emptyset$ and for each $\tau^\theta_{\mu_1}$ there are $O(1)$ $\tau^\theta_{\mu_2}$ with
$\tau^\theta_{\mu_1} \sim \tau^\theta_{\mu_2}$ and we also used \eqref{a.33}.
The proof of \eqref{a21} is complete.
\end{proof}

\noindent\textbf{Remark}:
For  Zoll surfacces, to prove \eqref{a3'}, let $\tilde P_k$ be defined as in \eqref{2pluss}. Then 
it suffices to show for $\tilde \rho_k=B\circ \rho_k$ with $\rho_k$ satisfying \eqref{a6}, we have
\begin{equation}\label{za13}
\|\sum_{k\in \mathbb{N}}\tilde\rho_k\tilde P_kf\|_{L^{\frac{14}{3}}_{t,x}(M\times [0,1])}\le C_\e \la^{\frac17+\e} \|f\|_{L^2(M)}, \,\,\,\forall \e>0.
\end{equation}

To prove this, one 
can define 
\begin{equation}\label{zolla}
\begin{aligned}
P_{k,\nu}=\tilde\rho_{k}A^{\theta_0}_{\nu} \tilde P_{k}f. 
\end{aligned}
\end{equation}
It is straightforward to check that Fourier support of $t$ variable is contained  in 
\begin{equation}\label{ztfourier}
    [(k+\frac{\alpha}4)^2-C_0, (k+\frac{\alpha}4)^2+C_0].
\end{equation}
Thus, by replacing the arguments that rely on $L^2$ orthogonality in the $t$ variable for $S^2$ with the {\em{almost}} $L^2$ orthogonality argument used in the proof of Lemma~\ref{mainL4},
one can repeat the arguments in this section to obtain \eqref{za13} for general Zoll surfaces.

\newsection{Sharpness of Theorem~\ref{main}}

To see that \eqref{Sn011} is sharp on any Zoll manifold, note that if we fix
 $\rho\in \mathcal{S}(\R)$ satisfying $\rho(0)=1$ and $\supp\hat\rho\subset (-\frac12, \frac12)$ as before, \eqref{Sn011} implies
\begin{equation}\label{Snsmoothe}
\|\rho(t)e^{-it\Delta_{g}}\eta(P/k)f\|_{L^{q}_{t,x}(M\times \R )}\le C_s k^{s} \|f\|_{L^2(M)}, \,\,\,\forall s>\mu(q),
\end{equation}
if $k\in\mathbb{N}^+$ and
\begin{equation}\label{3.2'}
\eta\in C^\infty_0((-1,1)) \quad
\text{with } \, \, \eta(t)=1, \, \, \, |t| \le 1/2.
\end{equation}
Now let us define $\chi_k$ to be the spectral projection operator for $\sqrt{-\Delta_g}$ onto the unit length interval
\begin{equation}\label{spe2}
    I_k=[k+\frac{\alpha-2}4,  k+\frac{\alpha+2}4].
\end{equation}
By \eqref{spe1}, the eigenvalues of $\sqrt{-\Delta_g}$ are clustered around a $A/k$ neighborhood of $k+\frac{\alpha}4$, thus it is straightforward to check that 
  the Fourier transform  of 
\begin{equation}\label{2plussa}
       \rho(t)e^{-it\Delta_{g}}\eta(P/k)\chi_k f .
    \end{equation}
in the $t$-variable is supported in 
\begin{equation}\label{tfouriera}
    [(k+\frac{\alpha}4)^2-C_0, (k+\frac{\alpha}4)^2+C_0],
\end{equation}
Thus, since these intervals are of bounded length, by Bernstein's inequality and \eqref{Sn011}, 
\begin{equation}\label{Snsmoothee}
\begin{aligned}
    \|\rho(t)e^{-it\Delta_{g}}\eta(P/k)\chi_kf\|_{L^\infty_tL^{q}_x(M\times \R )}&\le C\|\rho(t)e^{-it\Delta_{g}}\eta(P/k)\chi_kf\|_{L^{q}_{t,x}(M\times \R )}\\
    &\le C_s k^{s} \|f\|_{L^2(M)}, \,\,\,\forall  s>\mu(q).
\end{aligned}
\end{equation}
However, when $t=0$, for $k\ge C$ being large enough, we have
\begin{equation}\label{chi}
    \rho(t)e^{-it\Delta_{g}}\eta(P/k)\chi_k f= \eta(P/k)\chi_k f=\chi_k f,
\end{equation}
which yields 
\begin{equation}\label{chi1}
\|\chi_k f\|_{L^{q}_x(M)}\le C_s k^{s} \|f\|_{L^2(M)}, \,\,\,\forall s>\mu(q).
\end{equation}
As was shown in \cite{SFIO2}, if $\sigma(q)$ is defined as in \eqref{Sn}, the unit-band spectral estimate \eqref{chi1} fails for $ s<\sigma(q)$ on any
compact manifold, regardless of the geometry. Consequently, we have $\mu(q)\ge \sigma(q)$.

We shall also remark that in the special case of standard round sphere $S^n$, by letting $f=Q_\la$, the highest weight spherical harmonic with eigenvalue $\la$  for $ 2\le q\le \frac{2(d+1)}{d-1}$, and $f=Z_\la$, the zonal spherical function  for $q>\frac{2(d+1)}{d-1}$, we have 
\begin{equation}\label{sharp1}
    \|e^{-it\Delta_{g}}f\|_{L^q_{t,x}(I\times M)}=\|e^{it\la^2}f\|_{L^q_{t,x}(I\times M)}\approx \|f\|_{L^q_{x}(M)}\approx \la^{\sigma(q)}\|f\|_{L^2(M)}.
\end{equation}
This implies that \eqref{Sn011} can not hold for $s<\sigma(q)$. 
See e.g.,  \cite{blair2022improved} for more detailed calculations related to the $L^q$ norms of spherical harmonics.

It remains to show $\mu(q)\ge \frac d2-\frac{d+2}q$ if \eqref{Sn011} holds.
To see this,  fix $x_0\in M$, let $f_\la(x)=\la^{-d/2}\sum_j\beta(\la_j/\la)e_j(x)\overline{e_j(x_0)}$ for $\beta$ as in \eqref{00.11}, which is $\la^{-d/2} K(x,x_0)$ if $K(x,y)$ denotes the kernel of the multiplier operator $\beta(P/\la)$. For this choice of $f_\la$, we have 
\begin{equation}\label{sharp2}
\begin{aligned}
    \big|e^{-it\Delta_{g}}f_\la (x_0)\big|&=\la^{-d/2}\big|\sum_j\beta(\la_j/\la)e^{it\la_j^2}|e_j(x_0)|^2\big| \\
    &\approx \la^{d/2}\,\,\text{if}\,\,\,|t|\le \delta\la^{-2},
    \end{aligned}
\end{equation}
for some fixed small $\delta$. This follows from the fact that by the pointwise Weyl formula (see e.g., \cite{SoggeHangzhou}) 
$$\big|\sum_j\beta(\la_j/\la)|e_j(x_0)|^2\big|\approx \la^{d},\,\,\text{for any} \,\,x_0\in M,
$$
as well as 
$$ \text{Re} (e^{it\la_j^2})\ge 1/2\,\,\text{if}\,\,\,|t|\le \delta\la^{-2} \,\,\text{and}\,\,\,\la_j\approx\la.
$$
As a result, for any $2\le q\le \infty$,
\begin{equation}\label{sharp3}
    \|e^{-it\Delta_{g}}f_\la\|_{L^q_{t}L^\infty_{x}( M\times [0,1])}\gtrsim \|e^{-it\Delta_{g}}f_\la\|_{L^q_{t}L^\infty_{x}( M\times [0,\delta\la^{-2}])}\approx \la^{\frac{d}{2}-\frac2q}.
\end{equation}
If we let  $\tilde\beta\in C_0^\infty(\R)$ which equals 1 in a neighborhood of (1/2,2), then $\tilde\beta(P/\la)f_\la=f_\la$, and we have the Bernstein inequality
\begin{equation}\label{Berstein}
   \|\tilde\beta(P/\la)\|_{L^p\rightarrow L^q}\lesssim \la^{\frac{d}{p}-\frac{d}{q}},\,\,\,1\le p\le q\le \infty,
\end{equation}
which implies
\begin{equation}\label{sharp4}
\|e^{-it\Delta_{g}}f_\la\|_{L^q_{t}L^\infty_{x}( M\times [0,1])}\lesssim \la^{\frac{d}{q}}  \|e^{-it\Delta_{g}}f_\la\|_{L^q_{t,x}( M\times [0,1])}.
\end{equation}
On the other hand,  by $L^2$ orthogonality and the pointwise Weyl formula above, we have 
 \begin{equation}\label{l2}
     \|f_\la\|_{L^2(M)}\approx 1.
 \end{equation}
If we combine  \eqref{sharp3} \eqref{sharp4} and \eqref{l2}, we have
\begin{equation}\label{sharp5}
\|e^{-it\Delta_{g}}f_\la\|_{L^q_{t,x}( M\times [0,1])}\gtrsim \la^{\frac d2-\frac{d+2}q}\|f_\la\|_{L^2(M)},
\end{equation}
which implies that $s\ge \frac d2-\frac{d+2}q$.




\bibliography{refs1}
\bibliographystyle{abbrv}

%

\end{document}